%% file: SIAM_style_article.tex
\documentclass[hidelinks,onefignum,onetabnum]{siamart220329}

\input{preamble}

\ifpdf
\hypersetup{
  pdftitle={Riemannian optimization the symplectic Stiefel manifold using second--order information},
  pdfauthor={D. Doe, P. T. Frank, and J. E. Smith}
}
\fi

\ifpdf
  \DeclareGraphicsExtensions{.eps,.pdf,.png,.jpg}
\else
  \DeclareGraphicsExtensions{.eps}
\fi

\headers{Optimization on the symplectic Stiefel manifold}{R. Jensen and R. Zimmermann}

\title{Riemannian optimization on the symplectic Stiefel manifold using second--order information\thanks{Submitted to the editors 04/12 -- 2024.
\funding{This work was funded by the Independent Research Fund Denmark, DFF, grant nr. 3103-00094B}}}

\author{Rasmus Jensen\thanks{SDU Odense
  (\email{rasmusj@imada.sdu.dk}}
\and Ralf Zimmermann \thanks{SDU Odense
  (\email{zimmermann@imada.sdu.dk}}}

\usepackage{amsopn}

\usepackage{fbox}
\usepackage{algorithm}
\usepackage{algorithmic}

\usepackage{booktabs}
\newcommand{\ra}[1]{\renewcommand{\arraystretch}{#1}}
\usepackage{chngcntr}
\counterwithin{table}{section}

\usepackage{caption}
\DeclareCaptionFormat{cont}{#1 (cont.)#2#3\par}

\newcount\Comments  % 0 suppresses notes to selves in text
\newcommand{\kibitz}[2]{\ifnum\Comments=1\textcolor{#1}{#2}\fi}

\newcommand{\Nosp}{\textnormal{N}}
\usepackage{bookmark}

\begin{document}

\maketitle
% REQUIRED
\begin{abstract}
Riemannian optimization is concerned with problems, where the independent variable lies on a smooth manifold.
    There is a number of problems from numerical linear algebra that fall into this category, where the manifold is usually specified by special matrix structures, such as orthogonality or definiteness.
    Following this line of research, we investigate tools for Riemannian optimization on the symplectic Stiefel manifold.
    We complement the existing set of numerical optimization algorithms with a Riemannian trust region method tailored to the symplectic Stiefel manifold. To this end, we derive a matrix formula for the Riemannian Hessian under a right-invariant metric. Moreover, we propose a novel retraction for approximating the Riemannian geodesics. Finally, we conduct a comparative study in which we juxtapose the performance of the Riemannian variants of the steepest descend-, conjugate gradients- and trust region methods on selected matrix optimization problems that feature symplectic constraints.
\end{abstract}

% REQUIRED

% REQUIRED

\section{Introduction}

%\subsection{Problem statement and general introduction}
Riemannian optimization considers problems, where the feasible set $\M\subset \R^{d}$ is a Riemannian manifold. Specifically, constrained problems of the form
\begin{equation}\label{eq:min_prob}
    \min_{x\in \R^d}f(x),  \ \textnormal{such that} \ x\in \M,
\end{equation}
where $f:\R^d\to \R$ is a smooth function, 
formally become unconstrained problems 
\begin{equation}
    \min_{x\in \M}f(x)
\end{equation}
with smooth $f:\M \to \R$, when viewed from the intrinsic perspective of the manifold.
Manifolds have been used in numerical problems for several decades, see \cite{HelmkeMoore1994}.
In numerical linear algebra applications, maps $f$ from subsets of the matrix space $\R^{n\times m}\simeq \R^{nm}$ to $\R$ are of special interest. 
In the seminal paper \cite{Edelman}, the Stiefel and Grassmann manifolds are considered for their use in numerical linear algebra applications. In particular, the authors provide essential geometric tools for implementing Riemannian generalizations of the classical (Euclidean) nonlinear conjugate gradient-- and Newton methods. Since then, a substantial amount of literature has been produced, as well as textbooks; see for example \cite{absil_mahony_sepulchre_08, NBoumal,sato_book}.\\ 
Riemannian optimization techniques can be applied to problems such as the orthogonal Procrustes problem and the symmetric eigenvalue problem \cite{Edelman, absil_rtr_07}, the nonlinear eigenvalue problem \cite{nonlin_eig_newt_15}, low--rank matrix completion \cite{boumal_low_rank_15}, computation of the singular value decomposition \cite{Sato_13_svd}, multivariate statistics \cite{appl_multivar_20} and to obtain low--rank solutions of Lyapunov equations \cite{low_rank_lyapunow_10}. 

Optimization methods can be differentiated according to their order. First-order methods only use gradient information, while second--order methods also leverage the Hessian or suitable approximations of it.
In practice, this is a trade-off: for first-order methods, the individual iterations are usually inexpensive, but the number of iterations until numerical convergence is high. For second--order methods, the computational effort per iteration is comparatively high, but the total iteration count is low, see e.~g.  \cite{Noc_Wri}.

Methods in Riemannian optimization which exploit second--order information include the Newton method and trust-region method \cite{absil_rtr_07,Edelman}. Approximate second--order information is exploited in the nonlinear conjugate gradients method and quasi--Newton methods \cite{Edelman,Huang_broyden_15, Sato_22_article}.

\subsection{Original contributions}
This paper focuses on Riemannian optimization techniques on the symplectic Stiefel manifold  when equipped with a right--invariant Riemannian metric as in \cite{bz_symplectic}.
We complement the existing set of first--order and second--order optimization approaches by transferring the Riemannian trust region method to the manifold under consideration.

To this end, we derive a matrix formula for the Riemannian Hessian for smooth functions on the symplectic Stiefel manifold under the right-invariant metric. 
%\NEWVER{We also consider a possible approximation of the Riemannian Hessian, which is less computationally expensive.}
Moreover, we provide a new retraction based on the associated geodesics.\\
By means of numerical experiments, we juxtapose the Riemannian steepest descend, the Riemannian nonlinear conjugate gradients and the Riemannian trust region.
We consider the latter method using the exact metric Hessian operator and using the projected Euclidean Hessian operator.

We apply the various optimization methods to selected problems in numerical linear algebra.
More precisely, we consider the problems of computing the nearest symplectic matrix to a given matrix \cite{bz_symplectic}, computing the first $k$ symplectic eigenvalues of a symmetric positive definite matrix \cite{Son_21} and of computing the proper symplectic decomposition \cite{bz_symplectic,bendokat_geom_22}. 

\subsection{State of the art}
The symplectic group and the symplectic Stiefel manifold $\spst(2n,2k)$ are Riemannian manifolds that in addition feature the structure of a Lie group ($\sp(2n)$) and that of a quotient space $\spst(2n,2k)$, respectively. The two manifolds have been discussed in the context of Riemannian optimzation in \cite{Gao_2021, bz_symplectic}. Applications involving optimization on the symplectic group and the symplectic Stiefel manifold include structure--preserving model reduction of Hamiltonian systems \cite{Buchfink_22,bendokat_geom_22}, computing symplectic eigenvalues \cite{Son_21}, the averaging of linear optical systems \cite{Fiori_16} and quantum mechanics \cite{sp_quantum_06}.

So far, mainly first--order gradient descend approaches have been utilized \cite{bz_symplectic, Fiori_16, Gao_2021,Son_21}. The conjugate gradient method has been implemented in \cite{Yamada_Sato23}. Newton's method was discussed for the symplectic group in \cite{Birtea20}, but not for the symplectic Stiefel manifold. 
After we send this article to review, the independent preprint \cite{gao2024_newton} appeared. This paper differs from ours in several ways. For example, we use a quotient space representation of the symplectic Stiefel manifold, while \cite{gao2024_newton} treats it directly as an embedded submanifold. We employ the Christoffel formalism to derive the Riemannian Hessian, and the authors of \cite{gao2024_newton} derive the Riemannian Hessian based on a special operator viewpoint in close connection with so--called tractable Riemannian metrics. We, on the other hand, work under a right--invariant metric, and thus the Riemannian structures are not identical between the two works.

\subsection{Organization}
In Section \ref{sec:background}, we recap background theory on the symplectic Stiefel manifold. In Section \ref{sec:optimization}, we present three numerical algorithms for performing optimization, two of them known, one of them new, and discuss a few implementation details. In Section \ref{sec:num_exp}, we conduct numerical experiments, and finally we discuss our findings and conclude the paper in Section \ref{sec:discussion}.

%\section{Background theory}
\section{The symplectic Stiefel manifold}
\label{sec:background}
In this section we collect the necessary background theory for the symplectic Stiefel manifold. We complement the existing theory by deriving a representation for its normal space when viewed as a submanifold of $\R^{2n\times 2k}$, and we obtain a matrix formula for the orthogonal projection onto the tangent space. We introduce a new retraction, which is observed to numerically approximate the true geodesics better than the existing ones.
We discuss a low--rank formula for this retraction and discuss the computation of its directional derivative, which will eventually serve as a vector transport. Lastly, we provide a matrix formula for the Riemannian Hessian.

The symplectic Stiefel manifold is the set
\begin{equation}
    \spst(2n,2k)=\qty{U\in\R^{2n\times 2k}: U^TJ_{2n}U=J_{2k} },\ \ J_{2m}=\begin{bmatrix}
        0&I_m\\
        -I_m&0
    \end{bmatrix}. 
\end{equation}
It is a closed embedded smooth manifold in $\R^{2n\times 2k}$ of dimension $(4n-2k+1)k$, as shown for example in \cite{Gao_2021}. Briefly; consider the smooth map $h:\R^{2n\times 2k}\to \skew(2k)$, $U\mapsto h(U)=U^TJ_{2n}U-J_{2k}$, where $\skew(2k)$ denotes the set of all skew symmetric matrices. The set $h^{-1}(0)$ is a closed subset of $\R^{2n\times 2k}$ and the rank of the differential $\D h(U):\R^{2n\times 2k}\to \skew(2k)$ is $k(2k-1)$, since it is surjective onto $\skew(2k)$. This implies that $h$ is a submersion, which verifies the claim \cite{Gao_2021}.

It is not a compact manifold though, as it is unbounded; the diagonal matrix $D_n=\textnormal{diag}\qty{n,n^{-1}}, n\in\N$, is in $\spst(2,2)$, but $[D_n]_{11}$ can grow arbitrarily large with respect to any metric.

The symplectic Stiefel manifold is conveniently expressed as the set 
$$\spst(2n,2k)=\qty{U\in \R^{2n\times 2k}:U^+U=I_{2k}},$$
where $U^+=J_{2k}^TU^TJ_{2n}$ is called the \textit{symplectic inverse}. In the special case $n=k$, we recover the symplectic group $\sp(2n)=\qty{M\in \R^{2n\times 2n}:M^+M=I_{2n}}$. In $\sp(2n)$, the symplectic inverse $M^+$ is contained in $\sp(2n)$ since $ (M^+)^+M^+=MM^+$ and as $M^+M=I_{2n}\Leftrightarrow M M^+M=M$ implies that $MM^+=I_{2n}$. Since it holds for $M,N\in \sp(2n)$ that $(MN)^+MN=I_{2n}$, we have that indeed $\sp(2n)$ is a group. All elements of the symplectic group are invertible, and thus the column vectors of $M\in \sp(2n)$ span $\R^{2n}$. Likewise, elements in the symplectic Stiefel manifold feature full column rank, which follows from the fact that $U^TJ_{2n}U=J_{2k}$.

In \cite{bz_symplectic}, the authors consider the geometry of the symplectic Stiefel manifold under the quotient space representation $\spst(2n,2k)=\sp(2n)/\sp(2(n-k))$. We only provide a short review, and refer to \cite{bz_symplectic} for the details. The reader will find that many results are analogous to the corresponding results on the classical Stiefel manifold, see e.g. \cite{Edelman}. \\
For the symplectic group, the tangent space at $M\in\sp(2n)$ is 
\begin{equation}
    \T_M\sp(2n)=\qty{\Delta\in \R^{2n\times 2n}: \Delta^+M+M^+\Delta=0}, \\ \dim(\T_M\sp(2n))=(2n+1)n.
\end{equation}
Of course, the latter matches the dimension the symplectic group, 
$\dim(\sp(2n))=(2n+1)n$.
We denote the set of Hamiltonian matrices by  
$$\hamil(2n)=\qty{\Omega\in \R^{2n\times 2n}:\Omega^+=-\Omega},$$
and note that again $\dim(\hamil(2n))=(2n+1)n$. 
Taking $\Delta=M\Omega$, with $\Omega\in \hamil(2n)$, we have that $\Omega^++\Omega=0$ holds, showing that $\Delta=M\Omega$ for any $\Omega\in \hamil(2n)$ is a tangent vector. Taking instead $\Delta=\Omega M$ we have that $M^+\Omega^+M+M^+\Omega M=0\Leftrightarrow \Omega^++\Omega=0$, and thus $\Delta=\Omega M$ too is a tangent vector. Therefore,
\begin{align*}
    \T_M\sp(2n)&=\qty{M\Omega:\Omega\in\hamil(2n)}=\qty{\Omega M:\Omega\in\hamil(2n)}.
\end{align*}
Consider the bilinear form $g_M:\T_M\sp(2n)\times \T_M\sp(2n)\to\R, M\in \sp(2n)$
\begin{equation}\label{eq:rie_metric}
    g_M(X_1,X_2)=\frac{1}{2}\tr((X_1M^+)^TX_2M^+).
\end{equation}
This form is positive definite, thus a metric, since for $X=\Omega M\in \T_M\sp(2n)$, $$g_M(X,X)=\frac{1}{2}\tr(\Omega^T\Omega)>0 \quad \forall X\neq 0.$$ The metric $g_M$ varies smoothly with $M$, and it is  therefore a Riemannian metric on $\sp(2n)$. It is right--invariant in the sense that $g_{MN}(X_1N,X_2N)=g_{M}(X_1,X_2)$ for $M,N\in \sp(2n)$. The Riemannian metric induces a norm on the tangent space of $M$, $\|X\|_{M}=\sqrt{(g_M(X,X))}$. We will sometimes write $g$ instead of $g_M$ when the context is clear. 

\subsection{The symplectic Stiefel manifold as a quotient space}

In \cite{bz_symplectic}, the authors show that the quotient space $\sp(2n)/\sp(2(n-k))$ coincides with the symplectic Stiefel manifold $\spst(2n,2k)$. The quotient map $\pi:\sp(2n)\to\spst(2n,2k)$ is given by 
\begin{equation*}
    \pi(M)=ME=U, \ \ E=\begin{bmatrix}[c|c]
        I_k&0_k\\
        \hline
        0_{n-k\times k}&0_{n-k\times k}\\
        \hline
        0_k&I_k\\
        \hline
        0_{n-k\times k}&0_{n-k\times k}.
    \end{bmatrix},
\end{equation*}
in the sense that 
\begin{equation*}
    ME=
    \big[
        M^{(1:k)}\mid M^{(k+1:n-1)}\mid M^{(n:n+k)}\mid M^{(n+k+1:2n)}
    \big]
      E=
    \big[
        M^{(1:k)}\mid M^{(n:n+k)}
    \big]=U.
\end{equation*}
The tangent space of a quotient manifold $\mathcal{N}=\mathcal{M}/\mathcal{Q}$ can be identified with a certain subspace of the tangent space of the total manifold  $\mathcal{M}$.
The construction is as follows. 
Let $x\in \mathcal{N}$ and let $p\in\mathcal{M}$ be such that $\pi(p) = x$. 
Decompose the total tangent space as
\begin{equation}
    \T_p \M=\V_p^{\pi}\M\oplus \H_{p}^{\pi,g}\M,
\end{equation}
where $\V_p^{\pi}\M=\ker \qty(\D\pi(p))$ is the so--called vertical space, and $\D\pi(p)$ is the differential of the quotient map. The orthogonal complement $\H_p^{\pi,g}\M=\qty(\ker \qty(\D\pi(p)))^\perp$ is called the horizontal space.
After fixing a specific $p\in\mathcal{M}$ with $\pi(p)=x$, for any tangent vector $w\in \T_x\mathcal{N}$
there is one and only one horizontal tangent vector $v\in \H_p^{\pi,g}\M$ such that $\D\pi(p)[v] = w$, see   \cite{leesmooth}. Notice that the vertical space is independent of the metric, but the horizontal space is not, since the notion of orthogonality depends on the Riemannian metric $g$.

It is shown in \cite[p. 11]{bz_symplectic} that the horizontal space under the Riemannian metric in \eqref{eq:rie_metric} is 
\begin{equation}\label{eq:horizontal_space}
\begin{array}{rl}
    \H_M^{\pi,g}\sp(2n)&=\qty{\bar\Omega M:\bar\Omega=\Omega P+P\Omega-P\Omega P\in \hamil(2n)}, \\\
    P&=J_{2n}^TUU^+J_{2n}, \quad  U=\pi(M).
\end{array}
\end{equation}
The set of Hamiltonian matrices $\bar \Omega$ satisfying the condition in \eqref{eq:horizontal_space} is denoted by $\H_U\hamil(2n)$
and consequently 
\begin{equation*}
    \T_U\spst(2n,2k)=\qty{\bar \Omega U:\bar \Omega\in \H_U\hamil(2n)}. 
\end{equation*}
For each tangent $\eta\in \T_U\spst(2n,2k)$ and a specific representative $M$ of $U=\pi(M)$, there exists a unique horizontal lift $\eta_{\H_M^{\pi,g}}=\bar \Omega(\eta)M$, where
\begin{equation}
\label{eq:hor_lift_omega_eta}
    \bar\Omega(\eta)=\eta(U^TU)^{-1}U^T+J_{2n}U(U^TU)^{-1}\eta^T(I_{2n}-J_{2n}^TU(U^TU)^{-1}U^TJ_{2n})J_{2n}.
\end{equation}
Consequently, the Riemannian metric on $\T_U\spst(2n)$ is computed as
\begin{equation}\label{eq:rie_metric_spst}
    g_U(X_1,X_2)=\tr(X_1^T(I_{2n}-\frac{1}{2}J_{2n}^TU(U^TU)^{-1}U^TJ_{2n})X_2(U^TU)^{-1}).
\end{equation}
The formula is complicated by the symplectic structure, but is in close analogy to the corresponding formula for the canonical metric on the classical Stiefel manifold.
We refer the reader to \cite[p. 11--12]{bz_symplectic} for the details.
\subsection{Normal space and projections}
When viewing $\sp(2n)$ as an embedded submanifold of $\R^{2n\times 2n}$ endowed with the Riemannian metric \eqref{eq:rie_metric}, we can form the orthogonal complement of $\T_M\sp(2n)\subset \R^{2n\times 2n}$. This yields the  normal space $\Nosp_M\sp(2n)$, i.e., the space of vectors $\eta$ fulfilling $g_M(\Delta,\eta)=0,\forall \Delta\in \T_M\sp(2n)$ \cite[pp. 337]{leesmooth}.
Explicitly, we obtain
\begin{align*}
    \Nosp_M\sp(2n)&=\qty{\eta\in \R^{2n\times 2n}:g_M(\Delta,\eta)=0, \forall \Delta\in \T_M\sp(2n)}\\
    &=\qty{\eta\in \R^{2n\times 2n}:\tr(\Omega^T\eta M^+)=0,\forall \Omega\in \hamil(2n)}.
\end{align*}
Letting $\eta=SJ_{2n}^TM, S\in \skew(2n)$, a computation shows that $\tr(\Omega^T\eta M^+)=0$. Since $\T_M\sp(2n)\oplus \Nosp_M\sp(2n)=\R^{2n\times 2n}$, the dimension of $\Nosp_M\sp(2n)$ is $4n^2-2n^2-n=2n^2-n$, which matches the dimension of $\skew(2n)$, and hence we have 
\begin{lemma}
    For the symplectic group $\sp(2n)$ as an embedded submanifold of $\R^{2n\times 2n}$
    endowed with the right invariant metric \eqref{eq:rie_metric},
    the normal space, i.e., the orthogonal complement of the tangent space at $M\in \sp(2n)$, is 
    \begin{equation}
\Nosp_M\sp(2n)=\qty{\eta=SJ_{2n}^TM:S\in\skew(2n)}.
\end{equation}
Its dimension is $(2n-1)n$.
\end{lemma}
An orthogonal projection $P_M:\R^{2n\times 2n}\to \T_M\sp(2n)$ is a linear map which satisfies for any $V\in \R^{2n\times 2n}$ (1) $(P_M\circ P_M)(V)=P_M(V)$, (2) $P_M(\Delta)=\Delta, \ \forall \Delta\in \T_M\spst(2n,2n)$, (3) $\ker(P_M)=\Nosp_M\sp(2n)$, (4) $(P_M(V))^+M+M^+P_M(V)=0$ and (5) $g_M(P_M(V),\eta)=0, \ \forall \eta\in \Nosp_M\spst(2n,2n)$, where properties (4) and (5) follows from the others.

Writing $V=V_{T}+V_N$, where $V_T\in \T_M\sp(2n)$ and $V_N\in \Nosp_M\sp(2n)$, we have that $V-P_M(V)=SJ_{2n}^T M, S\in \skew(2n)$. From property (4) and the properties in Lemma \ref{lam:c1} we obtain
\begin{lemma}
The orthogonal projection $P_M:\R^{2n\times 2n}\to \T_M\sp(2n)$, $M\in\sp(2n)$, is given by 
    \begin{equation}\label{eq:proj_sp}
    P_M(V)=V-\frac{1}{2}M(V^+M+M^+V)=\frac{1}{2}V-\frac{1}{2}MV^+M.
\end{equation}

\end{lemma}
One can verify that properties (1) -- (3) and (5) are indeed satisfied. 

For $U\in \spst(2n,2k)$  and $\T_U\spst(2n,2k)$ endowed with the right--invariant metric \eqref{eq:rie_metric_spst}, we proceed along similar lines and obtain
\begin{lemma}
    For the symplectic Stiefel manifold $\spst(2n,2k)$ endowed with the right invariant metric \eqref{eq:rie_metric}, the normal space, i.e., the orthogonal complement of the tangent space at $U\in \spst(2n,2k)$, is
    \begin{equation}
        \Nosp_U\spst(2n,2k)=\qty{\eta =J_{2n}UT(U^TU):T\in \skew(2k)}.
    \end{equation}
    Its dimension is $(2k-1)k$. The orthogonal projection onto $\T_U\spst(2n,2k)$ is 
    \begin{equation}\label{eq:proj_spst}
        P_U(V)=V-\frac{1}{2}(U^T)^+((U^TU)^{+})^{-1}(V^+U+U^+V),
    \end{equation}
\end{lemma}
Observe that for $n=k$, the projection \eqref{eq:proj_spst} reduces to the projection in \eqref{eq:proj_sp}.
\subsection{Riemannian gradient of a scalar function on \texorpdfstring{$\spst(2n,2k)$}{SpSt(2n,2k)}} 

Given a smooth map $f:\spst(2n,2k)\to \R$, we have the defining relationship for the Riemannian gradient
\begin{equation}\label{eq:df_inp_rep}
    \D f(x)[v]=g(\grad f(x),v)=\inp{\nabla \bar f(x)}{v} =\D \bar f(x)[v],
\end{equation}
for a smooth extension $\bar f:\R^{2n\times 2k}\to \R$ of $f$, \cite[Definition 3.58]{NBoumal}. Here $\nabla \bar f(x)$ denotes the Euclidean gradient of $\bar f$ and $\inp{\cdot}{\cdot}$ denotes the canonical inner product on $\R^{2n\times 2k}$ $\inp{A}{B}=\tr(A^TB)$.  
In our applications, the smooth extension $\bar f$ can always be taken as $f$ with the input domain extended to all of $\R^{2n\times 2k}$, or at least to an open subset containing $\spst(2n,2k)$. Writing $\nabla f\equiv \nabla \bar f$, one can show that the Riemannian gradient is \cite{bz_symplectic}
\begin{equation}\label{eq:grad_spst}
    \grad f(U)=\nabla f(U) U^TU+J_{2n}U\nabla f(U)^TJ_{2n}U.
\end{equation}

\subsection{Geodesics and retractions}\label{sec:geod_ret}

In \cite{bz_symplectic}, the authors obtain the geodesics on $\spst(2n,2k)$ under the quotient space construction as the images of the horizontal geodesics on the symplectic group $\sp(2n)$: Given $M\in\sp(2n)$ and $\eta\in \T_M\sp(2n)$ the associated geodesic in the total space is
\begin{equation}\label{eq:geod_sp}
    \gamma(t)=\exp(t(\eta M^+-(\eta M^+)^T))\exp(t(\eta M^+)^T)M.
\end{equation}
If $\eta\in \H_M^{\pi,g}\sp(2n)$, then $\gamma'(t)\in \H_{\gamma(t)}^{\pi,g}\sp(2n)$ for all $t$, meaning that we have \textit{horizontal geodesics}. 
By a standard result on quotient spaces \cite[Corollary 7.46]{O_neil}, horizontal geodesics are mapped to geodesics in the quotient space under the quotient map $\pi$. Consequently, the geodesics on $\spst(2n,2k)$ are $\pi(\gamma(t))$. Given $U=\pi(M)$ and $\Delta\in \T_U\spst(2n,2k)$, one obtains
\begin{equation}\label{eq:geod_spst}
    \varphi(t)=\exp(t(\bar \Omega(\Delta)-\bar \Omega(\Delta)^T))\exp(t\bar \Omega(\Delta)^T)U.
\end{equation}
This follows from computing the horizontal lift of $\Delta$, $\Delta_{\H_M}^{\pi,g}=\bar \Omega(\Delta)M$, inserting in \eqref{eq:geod_sp} and composing with $\pi$.

The formulas \eqref{eq:geod_sp} and \eqref{eq:geod_spst} are not well suited for practical applications, because the matrix exponential of a Hamiltonian matrix is unstable.  In Riemannian optimization, it is standard practice to employ retractions rather than using the geodesics. A retraction is a map 
\begin{equation}
    R:\M\times \T\M\to \M, (x,v)\mapsto R_x(v),
\end{equation}
so that (1) $R_x(0)=x$ and  (2) $\D R_x(0)[\xi]=\xi$, $\forall \xi\in \T_x\M$ \cite[Definition 3.47]{NBoumal}. We can view retractions as first--order approximations of the Riemannian exponential map. By employing the following approximation 
\begin{equation}
    \exp(2X)\approx (I+X)(I-X)^{-1},
\end{equation}
we recover the so-called \textit{Cayley transformation} \cite[pp. 128]{HW}
%\RJ{So: We aim at citing results found in books as eg. \cite[Theorem 5.1]{NBoumal} or using the page number if more correct, but when referencing articles/parts of articles as a whole, eg. the CG method presented in Sato \cite{Sato_22_article}, we don't need to be as specific, except if it makes sense in the concrete case?}
%\RZ{In principle yes. If you reference a specific result (theorem, lemma, formula) then give a reference that uniquely identifies the result (could be eq. number, theorem number or a page number), and this also holds for research articles. If you reference for a general topic, say, CG, it is enough to just reference the paper. In case of a textbook, it would be enough to refer to a chapter or a section without precise page number. }
\begin{equation}
    \cay(X)=(I+X)(I-X)^{-1},
\end{equation}
which has the property that $\cay:\hamil(2n)\to \sp(2n)$ \cite{bz_symplectic}. Replacing the matrix exponentials in \eqref{eq:geod_spst} with the Cayley transformation yields the following \textit{Cayley retraction} 
\begin{equation}\label{eq:main_ret}
    R_U(\Delta)=\cay\qty(\frac{1}{2}\qty(\bar \Omega(\Delta)-\bar \Omega(\Delta)^T))\cay\qty(\frac{1}{2}\bar \Omega(\Delta)^T)U.
\end{equation}
To verify that this indeed is a retraction we note that $\bar \Omega(0)=0$ and thus $\cay(0)=I$, so indeed $R_U(0)=U$. Computing $\D R_x(0)[\xi]$ yields
\begin{align*}
    \D R_U(0)[\xi]&=\eval{\dv{x}R_U(0+x\xi)}_{x=0}\\
    &=\D \cay(0)\qty[\frac{1}{2}(\bar\Omega(\xi)-\bar\Omega(\xi)^T)]U+\D\cay(0)\qty[\frac{1}{2}\bar\Omega(\xi)^T]U.
\end{align*}
From the following calclation, using the fact that  $\D (I-X)[Y]=-Y$ and because $(I+X)$ commutes with $(I-X)^{-1}$
\begin{align}\label{eq:diff_cay}
    \D ((I-X)^{-1}(I+X))[Y]&=-(I-X)^{-1}(-Y)(I-X)^{-1}(I+X)+(I-X)^{-1}Y\notag\\
    &=(I-X)^{-1}\qty(Y(I+X)+Y(I-X))(I-X)^{-1}\notag \\
    &=2(I-X)^{-1}Y(I-X)^{-1}.
\end{align}
we readily obtain $\D R_U(0)[\xi]=\xi$.

The computation of $\cay(\bar \Omega(\Delta)^T)$ as well as that of $\cay(\bar \Omega(\Delta)-\bar \Omega(\Delta)^T)$ requires inversion of a $2n\times 2n$ matrix. To avoid this, we proceed in a similar manner as in \cite{bz_symplectic} and \cite{Gao_2021} and obtain an efficient formula for the retraction. We mention that the work \cite{Oviedo_2023} contains theory on a general family of efficient retractions, but to the best of our knowledge, the retraction in \eqref{eq:main_ret} does not fit in the proposed family. 

Without further comments we use the result appearing in \cite[p. 13]{bz_symplectic}  that every tangent vector $\Delta\in \T_U\spst(2n,2k)$ can be written as $\bar \Delta=UA+H$ where 
\begin{equation*}
A=J_{2k}U^T\Delta(U^TU)^{-1}J_{2k}+(U^TU)^{-1}\Delta^TU-(U^TU)^{-1}\Delta^TJ_{2n}^TU(U^TU)^{-1}J_{2k}
\end{equation*}
\begin{equation*}
    H=(I_{2n}-UU^+)J_{2n}\Delta(U^T U)^{-1}J_{2k}.
\end{equation*}
Having the two components defined above, we can calculate, for $\Delta=\bar \Omega(\Delta)U$, the matrix $\bar \Omega(\Delta)$, as 
\begin{equation*}
    \bar \Omega(\Delta)=\begin{bmatrix}
        J_{2n}^TUJ_{2k}&(\bar\Delta^+\qty(I_{2n}-\frac{1}{2}UU^+))^T
    \end{bmatrix}\begin{bmatrix}
        \bar \Delta^T\qty(I_{2n}-\frac{1}{2}UU^+)^T\\
        -U^T
    \end{bmatrix}=: YX^T,
\end{equation*}
where $X,Y\in \R^{2n\times 4k}$. Furthermore, by setting $\hat X=\begin{bmatrix}
    Y&-X
\end{bmatrix}$ and $\hat Y=\begin{bmatrix}
    X&Y
\end{bmatrix}$, we have that $\bar\Omega(\Delta)-\bar\Omega(\Delta)^T=\hat X\hat Y^T$. 
%\RZ{General comment: I would likely not spot any typos in the formulas, simply because they are so involved. For doing so, I would need to redo all the calculations, which I want to avoid at this stage.}

From the construction above, we can mimic the ideas presented in \cite{bz_symplectic} leading to low--rank formulas for computing the retraction \eqref{eq:main_ret}. We refer the reader to Appendix \ref{sec:appendixB} for the derivation.  One obtains
\begin{equation}\label{eq:eff_ret}
    \begin{aligned}
    R_U(t\Delta)=&\qty(I_{2n}+t\hat X\qty(I_{8k}-\frac{t}{2}\hat Y^T\hat X)^{-1}\hat Y^T)\\
    &\cdot \qty(-U+(tH+2U)\qty(I_{2k}-\frac{t}{2}A+\frac{t^2}{4}H^+H)^{-1}).
\end{aligned}
\end{equation}
Notice that the formula \eqref{eq:eff_ret} is only efficient when $4k <n$.

One can also consider the (simpler) retraction \cite{bz_symplectic}
\begin{equation}\label{eq:simple_ret}
    \tilde R_U(t\Delta)=\cay\qty(\frac{t}{2}\bar\Omega(\Delta))U,
\end{equation}
which also allows for efficient evaluation due to the result in Equation  \eqref{eq:efficient_cay_1}. A reason for choosing $\tilde R$ over $R$ is that it is computationally cheaper to evaluate. However, $R$ is a better approximation of the true geodesics. This is confirmed by a numerical experiment presented in Figure \ref{fig:num_exp_geod_ret}, and therefore we will use \eqref{eq:eff_ret} as the retraction in our numerical experiments. 
\begin{figure}[!ht]
    \centering
    \includegraphics[scale = 0.4]{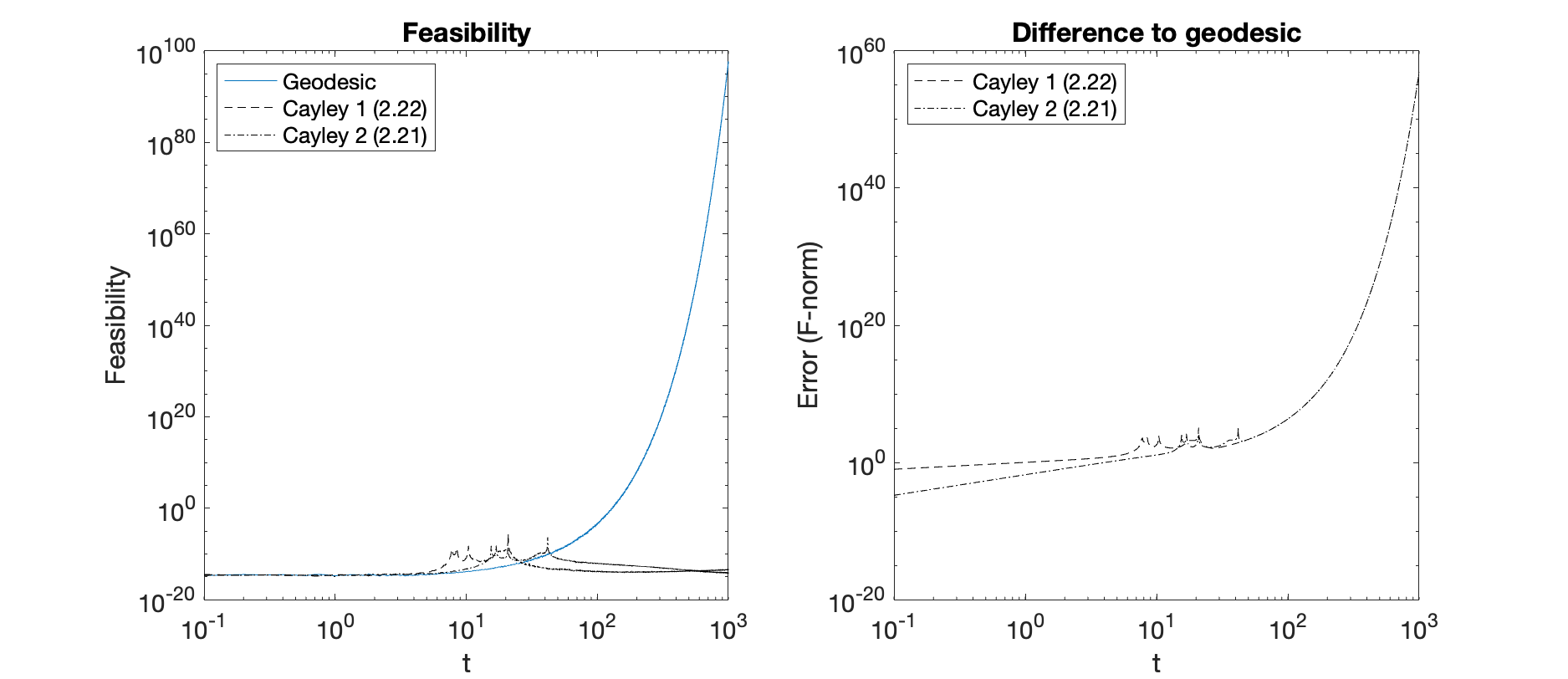}
    \caption{Numerical experiment with random point $X\in \spst(2n,2k)$ and random $\Delta\in \T_X\spst(2n,2k)$. Here $n = 100$ and $k = 10$. Left: Feasibility quantified by $\|\gamma(t)^+\gamma(t)-I_{2k}\|_F$ for the three different methods of moving on $\spst(2n,2k)$. The geodesic is computed according to \eqref{eq:geod_spst}, Cayley 1 is \eqref{eq:simple_ret} and Cayley 2 is \eqref{eq:eff_ret}.  Clearly, the computed geodesic does not remain on the manifold for large $t$. Right: The error measured in the Frobenious norm of the difference between the geodesic and the two retractions respectively.  
    } \label{fig:num_exp_geod_ret} 
\end{figure}
\begin{figure}[!ht]
    \ContinuedFloat
    \captionsetup{list=off,format=cont}

    \includegraphics[scale = 0.4]{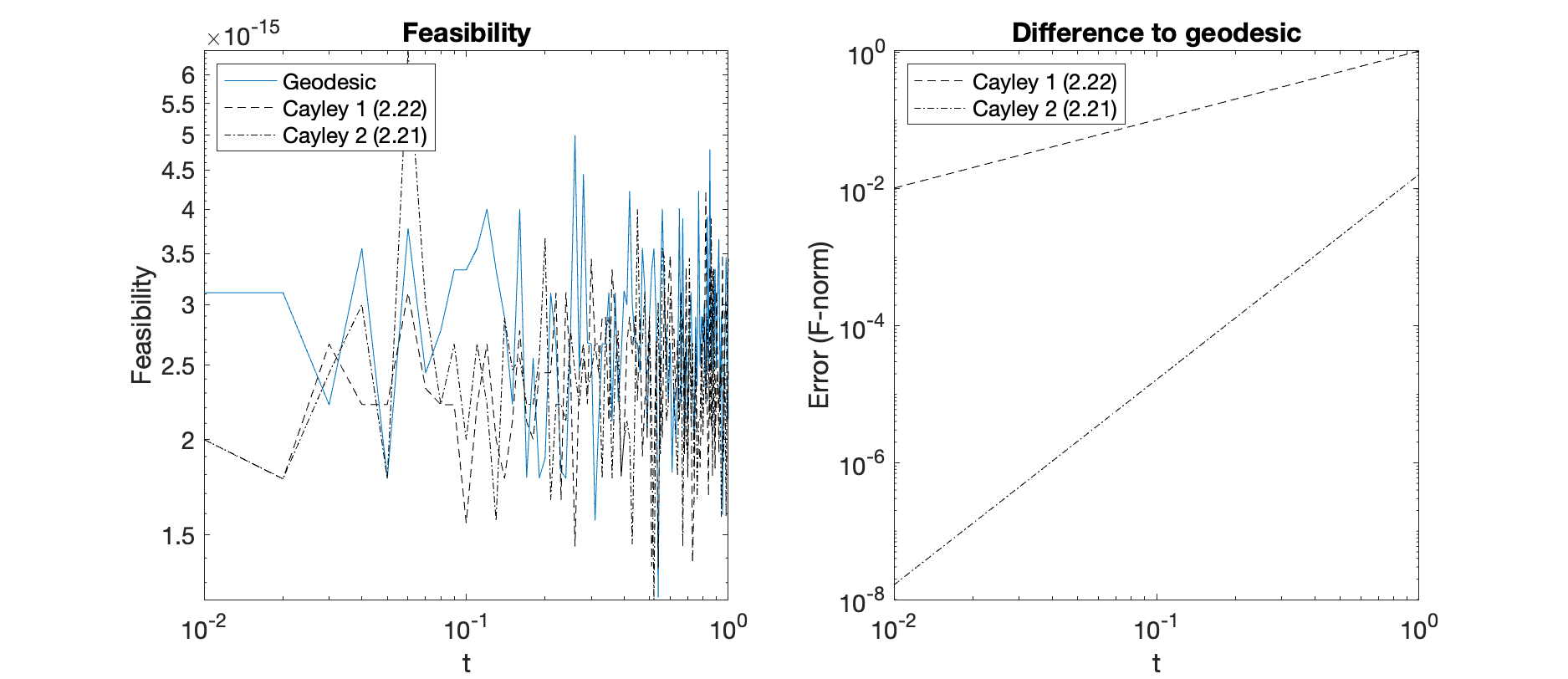}
    \caption{Left: Continuation of the numerical experiment in Figure \ref{fig:num_exp_geod_ret} on $\spst(2n,2k), n=100,k=10$. Zooming in on the interval $[0.01,1]$ for $t$, both Caley retractions stay on the symplectic Stiefel manifold up to numerical errors of the order $10^{-15}$. Right: The retraction \eqref{eq:eff_ret} is approximating the true geodesic up to an error that is several orders of magnitude lower than that of \eqref{eq:simple_ret}. }
    \label{fig:num_exp_geod_ret_2} 
\end{figure}
\subsection{Vector transport on \texorpdfstring{$\spst(2n,2k)$}{SpSt(2n,2k)}} 

Moving tangent vectors between tangent spaces of a Riemannian manifold $\M$ requires a so--called vector transport \cite{sato_book}
\begin{equation*}
    \tran:\T\M\times \T\M\to \T\M.
\end{equation*}
Vector transport can be considered as a first--order approximation to the Riemannian parallel transport \cite[Chapter 2.3]{docarmo92} analogous to retractions being first--order approximations of the Riemannian exponential map.
Vector transport is only preferred over Riemann parallel transport in applications if the associated computations are either cheaper or more stable.\\
Given a retraction $R$, a point $x\in \M$ and three tangent vectors $\eta,\xi,\phi\in \T_x\M$, a vector transport must satisfy  \cite[Definition 10.68]{NBoumal}
\begin{enumerate}
    \item $\tran_\eta(\xi)\in \T_{R_x(\eta)}\M$ (The vector $\xi$ is transported to a vector determined by the direction $\eta$ and the retraction $R_x$),
    \item $\tran_{0_x}(\xi)=\xi$ (The vector $\xi$ is left unchanged whenever $R_x(0_x)=x$) and 
    \item $\tran_\eta(a\xi+b\phi)=a\tran_\eta(\xi)+b\tran_\eta(\phi)$ (linearity)
\end{enumerate}
A standard way to construct a vector transport is to use the \textit{differential of a retraction} \cite{NBoumal}
\begin{equation*}
    \D R_x(\eta):\T_x\M\to \T_{R_x(\eta)},\  \xi\to \D R_x(\eta)[\xi].
\end{equation*}
Alternatively, the so--called \textit{inverse retraction} \cite{sato_zhu} can be utilized, if available. 
%The vector transport is a first order approximation of Riemannian parallel transport, which often doesn't feature closed form. It is, for example, available for the orthogonal group, but not for the Stiefel manifold \cite{Edelman}. 
We can also employ the orthogonal projections \eqref{eq:proj_sp} and \eqref{eq:proj_spst}; given a point $x\in \spst(2n,2k)$, $P:\R^{2n\times 2k}\to \T_x\spst(2n,2k)$, which yields condition 1 in the definition of a vector transport, $P(V)=V$ if $V\in \T_x\spst(2n,2k)$, which is condition 2.
Finally, since projections are linear, condition 3 is also satisfied. The Riemannian parallel transport preserves the norm of the vector that is transported. Any vector transport can be enhanced to feature this property by a suitable scaling.

\begin{lemma}\label{lem:diff_ret}
    The directional derivative of the Cayley retraction \eqref{eq:main_ret} is 
    \begin{equation}
        \begin{aligned}
    \D R_U(t\eta)[\xi]=&\qty(I_{2n}-\Omega_1(\eta))^{-1}\Omega_1(\xi)\qty(I_{2n}-\Omega_1(\eta))^{-1}\cay(\Omega_2(\eta))U\\
    &+\cay\qty(\Omega_1(\eta))(I_{2n}-\Omega_2(\eta))^{-1}\Omega_2(\xi)(I_{2n}-\Omega_2(\eta))^{-1}U,
\end{aligned}
    \end{equation}
where $\Omega_1(\eta)=\frac{t}{2}\qty(\Omega(\eta)-\Omega(\eta)^T)$, $\Omega_2(\eta)=\frac{t}{2}(\Omega(\eta))^T$, $\Omega_1(\xi)=\frac{1}{2}\qty(\Omega(\xi)-\Omega(\xi)^T)$ and $\Omega_2(\xi)=\frac{1}{2}(\Omega(\Delta))^T$
\end{lemma}
\begin{proof}
    For the retraction in \eqref{eq:main_ret}, we obtain 
    \begin{align*}
        \D R_U(t\eta)[\xi]=&\D \cay\qty(\frac{t}{2}(\Omega(\eta)-\Omega(\eta)^T))[\xi]\cay\qty(\frac{t}{2}\Omega(\eta)^T)U\\
        &+\cay\qty(\frac{t}{2}(\Omega(\eta)-\Omega(\eta)^T))\D\cay\qty(\frac{t}{2}\Omega(\eta))[\xi]U,
    \end{align*}
    which can be computed using the formula in \eqref{eq:diff_cay}, 
\end{proof}

The formula in Lemma \ref{lem:diff_ret} is costly to evaluate, but the effort can be reduced by leveraging the low-rank structure and applying the techniques used in Appendix \ref{sec:appendixB}. Moreover, if one first computes the retraction $R_U(t\eta)$ and then subsequently calculates the vector transport $\D R_U(t\eta)[\eta]$, then one can recycle various components. This is exactly the case for the Riemannian conjugate gradients method, and thus we make use of this recycling of computed quantities in our numerical experiments.

\subsection{The Riemannian Hessian of a scalar function on \texorpdfstring{$\spst(2n,2k)$}{SpSt(2n,2k)}}\label{sec:rie_hess}
%\subsection{The Riemannian Hessian for $\spst(2n,2k)$}
Computing the Riemannian Hessian of a function $f:\M\to \R$ can be done by computing the covariant derivative  of the Riemannian gradient $\grad f$, along a geodesic $\varphi(t)$. Alternatively, one can efficiently read off the so--called Christoffel form (a bilinear 2--1 tensor) from the geodesic differential equation, and use this in turn to compute the Riemannian Hessian. We will follow that path, and refer the reader to Appendix \ref{sec:appendix} for a brief summary of the theory. We derive the Riemannian Hessian for a scalar function on the symplectic Stiefel manifold under the Riemannian metric $g$ in \eqref{eq:rie_metric} (and, as an aside, for the symplectic group as well).

From the geodesic formula in \eqref{eq:geod_spst} we have, by setting $\Omega:=\bar\Omega(\Delta)$
\begin{align*}
    \varphi'(t)=&(\Omega-\Omega^T)\exp(t(\Omega-\Omega^T))\exp(t\Omega^T)U\\
    &+\exp(t(\Omega-\Omega^T))\Omega^T\exp(t\Omega^T)U,\\
    \Rightarrow \varphi''(0)=&(\Omega-\Omega^T)(\varphi'(0)+\Omega^TU)+(\Omega^T)^2U.
\end{align*}
By \eqref{eq:geod_christ_rel} in the appendix, this yields the Christoffel symbols, and consequently the Riemannian Hessian. 
\begin{lemma}
\label{lem:ChristoffelForm}
Let $U\in \spst(2n,2k)$ and let $\Delta \in T_U\spst(2n,2k)$ be a tangent vector and let $\bar\Omega(\Delta)$ be as in \eqref{eq:hor_lift_omega_eta}.
The Christoffel symbols $\Gamma_g$ with respect to the right-invariant metric $g$ of \eqref{eq:rie_metric} on $\spst(2n,2k)$ are given by
\begin{equation}\label{eq:gamma_g_same_input}
    \Gamma_g(\Delta,\Delta)=-(\bar\Omega(\Delta)-\bar\Omega(\Delta)^T)(\Delta+\bar\Omega(\Delta)^TU)-(\bar\Omega(\Delta)^T)^2U.
\end{equation}
%\RZ{Open question: do we want to include specifics about how to calculate $\eval{\dv{t}}_{t=0}\grad f(c(t))$}
The Hessian at $U$ of a smooth scalar function $f: \spst(2n,2k) \to \R$ under the metric \eqref{eq:rie_metric} is the endomorphism $\hess f(U): \T_U\spst(2n,2k) \to \T_U\spst(2n,2k)$,
\begin{equation*}
   \hess f(U)[\Delta]=\eval{\dv{t}}_{t=0}\grad f(c(t))+\Gamma_g(\grad f(U),\Delta),
\end{equation*}
where $c:t\mapsto c(t)$ is an arbitrary curve on $\spst(2n,2k)$ such that $c(0)=U, c'(0)=\Delta$ and 
$\eval{\dv{t}}_{t=0}\grad f(c(t))$ is the (standard) derivative of the Riemannian gradient of $f$.
\end{lemma}
Considering the general formula for the Riemannian gradient in \eqref{eq:grad_spst} one can write down an expression for $\eval{\dv{t}}_{t=0}\grad f(c(t))$, which also features in our implementation.

A symbolic formula for the Christoffel symbols for general inputs $\Gamma_g(\Delta,\tilde\Delta)$ can be obtained by applying the polarization identity \eqref{eq:polar}, but this does not lead to a computational reduction.

%\RJ{I suggest that the line above is added, and that the text below is removed, as it is 1) makes little sense to a peer 2) seems obsoloete for a research paper. }
%
%\RJ{However, we found that the number of terms in this expression remains large and we also did not find computational benefits in this approach. In practice, we will first evaluate terms of the form of \eqref{eq:gamma_g_same_input} and then plug them in in \eqref{eq:polar} when computing the Hessian.}

Practical tests revealed that it is numerically beneficial to perform a certain normalization procedure. By definition, and also readily seen in eq. \eqref{eq:christ_form}, the Christoffel form is bilinear, and thus we can write 
\begin{equation*}
    \Gamma_g(\Delta,\Delta)=\|\Delta\|_F^2\Gamma_g\qty(\frac{\Delta}{\|\Delta\|_F}\frac{\Delta}{\|\Delta\|_F}).
\end{equation*}
and use this when we compute $\Gamma_g(\grad f(U),\Delta)$ via the polarization identity. It is our experience that this stabilizes the computation significantly, and ensures that $\hess f(U)[\Delta]\in \T_U \spst(2n,2k)$, i.e., the the output of the Hessian stays on the tangent space numerically.

The Riemannian Hessian is computationally costly to evaluate, relative to the Riemannian gradient and other quantities, used in practical algorithms. This is is seen by noting that $\bar\Omega(\Delta)\in \R^{2n\times 2n}$, and as the formulas do not reduce due to nonlinearity, there seems to be little room for improvement, unless one works with numerical approximations.

\subsection{Approximation of the Riemannian Hessian}
The optimization literature knows a large variety of methods for incorporating approximate second-order information without calculating the full Hesse operator. As just one example, we remind the reader of the various forms (Polak-Ribi\'ere, Fletcher-Reeves, etc.) for computing the $\beta$-coefficients in nonlinear CG.
In the special case of a Riemannian submanifold endowed with the Eucludean metric, the Riemannian Hesse 1--1 form of a smooth map $f:\M\to\R$ is obtained from a projection, \cite[Corrolary 5.16]{NBoumal}
\begin{equation}\label{eq:ehess_to_rhess_euc}
    \hess f(x)[\Delta]=P_x\qty(\D \overline \grad f (x)[\Delta]), \ \ \Delta\in \T_x\M,
\end{equation}
where $\overline \grad f$ is a smooth extension of $\grad f$. 
Obviously, \eqref{eq:ehess_to_rhess_euc} is cheaper to obtain than the Hessian operator associated with the right-invariant metric, while it still can be expected to encode valuable second-order information.
In the experiments we will consider it as an approximation of the true Riemannian Hessian.

\section{Optimzation algorithms on Riemannian manifolds}\label{sec:optimization}

In this section, we consider problems from numerical linear algebra which all can be approached via Riemannian optimization methods on $\spst(2n,2k)$. 
Our goal is to conduct a performance study for the Riemannian variants of the steepest descend method {\bf (R--SD)}, the nonlinear conjugate gradients method {\bf (R--CG)} and the trust region method {\bf (R--TR)}. We have chosen to leave out Newton's method, as it would require to solve large equation systems with the Hessian as the operator. We also tested the Riemannian limited--memory BFGS method \cite{Huang2016}, which is implemented in the Manopt package for Matlab \cite{manopt}, which is a general purpose toolbox for Riemannian optimization. This method avoids to compute the true Hessian, but rather it relies on a special approximation. In our initial tests, we did not see the BFGS method to be competitive when compared with the other three methods under consideration. We believe that much customization would be required. Thus we chose to not include it in our experiments.

\subsection{Riemannian steepest descend and nonlinear conjugate gradients methods}

Let $\M$ is a Riemannian manifold. The Riemannian steepest descend and nonlinear conjugate gradients methods on $\M$ are straightforward generalizations of their Euclidean counterparts. We state them in a combined form in Algorithm \ref{alg:cg_nonlin_M}. Here and throughout, we write $f_k:=f(x_k)$ for the objective function value at iteration $k$.  
\begin{algorithm}
    \caption{The Riemannian nonlinear conjugate gradients method with restart \cite{Sato_22_article}}    
    \label{alg:cg_nonlin_M}
    \begin{algorithmic}[1]
        \REQUIRE Objective function $f:\M\to \R$, vector transport $\mathcal{T}^{(k)}$ , initial guess $x_0\in \M$, retraction $R:\T \M\to \M$,  $\mu\in \N$ and initial search direction $p_0=-\grad f_0$
        \FOR{$k=1,2,\dots$}
            \STATE Calculate $\tau_k=\argmin_{\tau>0}f(R_{x_k}(\tau p_k))$
            \STATE Set $x_{k+1}=R_{x_k}(\tau_kp_k)$
            \STATE Compute $\beta_{k+1}$ as for example $\beta_{k+1}=\frac{g_{x_{k+1}}(\grad f_{k+1},\grad f_{k+1})}{g_{x_k}(\grad f_k,\grad f_k)}$ (Fletcher--Reeves)
%            Compute the next search direction: 
            \IF{$k=\mu z, z\in \N$}
               \STATE $p_k=-\grad f(x_k)$
            \ELSE
                \STATE $p_{k+1}=-\grad f_{k+1}+\beta_{k+1}\mathcal{T}^{(k)}(p_k)$ where $\mathcal{T}^{(k)}:\T_{x_k}\M\to \T_{x_{k+1}}\M$ is isometric
            \ENDIF
        \ENDFOR 
        
        \end{algorithmic}
\end{algorithm}

Well--established theory on the convergence of Algorithm \ref{alg:cg_nonlin_M} for both the Euclidean and the Riemannian setting exists, and the reader can consult e.g. \cite[Chapter 5]{Noc_Wri} for the Euclidean variant and \cite{Sato_22_article} for the Riemannian variant. If line 4 of Algorithm \ref{alg:cg_nonlin_M} is replaced with $\beta_{k+1}=0, \forall k$, or if $\mu=1$ is chosen, 
the algorithm reduces to R--SD. Convergence theory for R--SD is presented in e.g. \cite[Chapter 4]{NBoumal}. If $\mu>1$ and $\beta_{k+1}\neq 0$, \Cref{alg:cg_nonlin_M} conducts R--CG. In our experiments we take $\beta_{k+1}=\frac{g_{x_{k+1}}(\grad f_{k+1},\grad f_{k+1})}{g_{x_{k}}(\grad f_{k},\grad f_{k})}$, which is the so--called Fletcher--Reeves type of the coefficient $\beta_{k+1}$. It is the introduction of $\beta_{k+1}\neq 0$ which makes R--CG differ from R--SD, because information on the old search direction are used when forming the next one. In the Euclidean setting, linear CG features exact second--order information, because then $\beta_{k+1}$ is a ratio of the Hessian of the objective function at $x_k$ and $x_{k+1}$. In the nonlinear and in the Riemannian setting $\beta_{k+1}$ is a ratio of finite--difference approximations of the Hessian. There exists numerous  choices of $\beta_{k+1}$, and the reader can find a recent overview of these in \cite{Sato_22_article}. To avoid jamming of R--CG, we employ periodic restarts, where the search direction is re-taken as the steepest descend direction, and buildup of bad information and numerical errors is removed. One can use a modified Polak--Rib\`ere version of $\beta_{k+1}$ introduced in \cite{Sato_22_article} to overcome this, but it requires the computation of a vector transport, which  introduces further computational cost. The assumptions one has to impose on R--SD is less restrictive than those one has to impose on R--CG in order to guarantee convergence, and R--CG does usually only converge in the sense that $\liminf_{n\to \infty}\|\grad f(x_k)\|_{x_k}=0$, whereas for R--SD one obtains, under suitable conditions, $\lim_{n\to\infty}\|\grad f(x_k)\|_{x_k}=0$. In general, to ensure convergence theoretically, one imposes, among several regularity conditions, the so--called Riemannian Wolfe conditions on the line search carried out in line 2 of Algorithm \ref{alg:cg_nonlin_M}
\begin{align}
    f(R_{x_k(\tau_kp_k)})&\leq f(x_k)+c_1\tau_k g_{x_k}(\grad f(x_k),p_k) \label{eq:armijo_cond}\\
    g_{x_{k+1}}(\grad f(R_{x_k}(\tau_k p_k)), \mathcal{T}^{(k)}(p_k))&\geq c_2g_{x_k}(\grad f(x_k),p_k),\label{eq:wolfe_2}
\end{align}
or possibly stronger conditions, depending on the choice of $\beta_{k+1}$  \cite{Sato_22_article}. The condition \eqref{eq:armijo_cond} is called the Armijo condition, and in practice we only aim to satisfy that condition, because satisfying both parts of the Wolfe conditions is computationally costly. We stress that in either case, the line search is inexact.  

We impose norm-preservation on the vector transport $\mathcal{T}^{(k)}$, i.e., for $\eta\in \T_{x_k}\M$, $\|\mathcal{T}^{(k)}(\eta)\|_{x_{k+1}}=\|\eta\|_{x_k}$. 
This is in line with the Riemannian parallel transport, which naturally features this property.
%This is a simplification compared to a flexible scaling parameter allowed in \cite{Sato_22_article}. 
%Enforcing this is straightforward, since we can just scale the vector transport with $\frac{\|\eta\|_{x_k}}{\|\mathcal{T}^{(k)}(\eta)\|_{x_{k+1}}}$. 
In \cite{Sato_22_article}, it is required that the vector transport approximates the directional derivative of the retraction in order to guarantee convergence.

In our numerical set-up, the line search is conducted via the Barzilai--Borwein method {\bf (BB--method)}, which we have adapted from \cite{gao2022optimization}. 
\begin{algorithm}
    \caption{The Barzilai--Borwein line search method. This algorithm carries out step 2 and 3, in each iteration $i$, of Algorithm \ref{alg:cg_nonlin_M}.}
        \label{alg:BB_method}
    \begin{algorithmic}[1]
        \REQUIRE Objective function $f:\M\to \R$ , retraction $R:\T \M\to \M$, search direction $p_i$, $\gamma_0>0$ $0<\gamma_{\min}<\gamma_{\max}$, $\beta,\delta\in (0,1)$, $\alpha\in [0,1]$, $q_0=1$, and $c_0=f(x_0)$
        \IF{$i>0$}
            \STATE Compute
            \begin{equation*}
                \gamma_i=\begin{cases*}
                    \frac{\|W_{i-1}…\|_F^2}{|\tr(W_{i-1}^TY_{i-1})|},& \textnormal{if $i$ odd}\\
                    \frac{|\tr(W_{i-1}^TY_{i-1})|}{\|Y_{i-1}\|_F^2},&\textnormal{if $i$ even}
                \end{cases*}
            \end{equation*} 
            where $W_{i-1}=x_i-x_{i-1}$ are the two last iterates, and $Y_{i-1}=Z_i-Z_{i-1}$, with $Z_j=\grad f(x_j)$.
        \ENDIF
        \STATE Compute the starting guess $\gamma_i=\max\qty{\gamma_{\min},\min\qty{\gamma_i,\gamma_{\max}}}$\\
        \STATE Compute the smallest integer $l$ so that 
        \begin{equation*}
            f(R_{x_i}(\tau_i p_i))\leq f(x_i)+\beta \tau_i g_{x_i}(\grad f(x_i),p_i),
        \end{equation*}
        where $\tau_i=\gamma_i\delta^l$\\
        \STATE Update $x_{i+1}=R_{x_i}(\tau_i p_i)$\\
        \STATE Set $q_{i+1}=\alpha q_i+1$ and $c_{i+1}=\frac{\alpha q_i}{q_{i+1}}+\frac{1}{q_{i+1}}f(x_{i+1})$
        \end{algorithmic}
\end{algorithm}
According to our observations, the BB method is to be preferred over a simple backtracking scheme, since the starting guess at each iteration is selected based on information from the previous iterations. The BB method was also used in other works, eg. \cite{Gao_2021, bz_symplectic,Oviedo_2023}. In the line search, we only verify that the Armijo condition \eqref{eq:armijo_cond} is satisfied. In the experiments, we take $\gamma_0=f(x_0)$ for an initial guess $x_0\in \M$, $\gamma_{\min}=10^{-15}, \gamma_{\max}=10^{15},\beta=10^{-4},\delta=10^{-1},\alpha =0.85$. For R--CG we take $\mu=5$ in Algorithm \ref{alg:cg_nonlin_M}.

\subsection{Riemannian trust--region method}

Like the two preceding algorithms, the Riemannian trust--region method is too a Riemannian generalization of an iterative method for linear spaces. Trust region methods utilize both first-- and second--order information, where at each iteration a quadratic approximation of the objective function $f:\M\to \R$, the so--called subproblem, is minimized over the tangent space $\T_{x_k}\M$
\begin{equation}\label{eq:quadratic_model}
    f(R_{x_k}(\xi))\approx  m_k(\xi)=f(x_k)+g_{x_k}(\grad f(x_k), \xi)+\frac{1}{2}g_{x_k}(H_k(\xi),\xi).
\end{equation}
Here, $\xi\mapsto H_k(\xi)$ is a linear map $H_k:\T_{x_k}\M\to \T_{x_k}\M$ which either can be taken as the Hessian of $f$, $\hess f(x_k)[\xi]$, or as an approximation of the Hessian \cite[Chapter 6]{NBoumal}, \cite{absil_rtr_07}. In our experiments, we work with $H_k(\xi)=\hess f(x_k)[\xi]$, and $\hat H_k(\xi)=P_{x_k}(\D \overline \grad f(x_k)[\xi])$.

The procedure is summarized in Algorithm \ref{alg:r_tr}. It is based on \cite[Algorithm 6.3]{NBoumal} and the implementation in Manopt. The user parameters $Q_0,\bar Q$  and $\rho$ of this algorithm are taken as $\mathcal{Q}_0=\frac{\bar{\mathcal{Q}}}{8}$, $\bar{\mathcal{Q} }=\sqrt{\dim \M}$ and $\rho'=0.1$. Solving the subproblem in line 2 is done via the so--called truncated conjugate gradients method, which is an optimized conjugate gradients method for trust region methods. We will not go into details, but these can be found in \cite[Chapter 6]{NBoumal} or \cite{absil_rtr_07} and references therein. We use the standard parameters provided in the manopt implementation.
\begin{algorithm}
    \caption{The Riemannian trust--region method \cite{NBoumal}}
        \label{alg:r_tr}
    \begin{algorithmic}[1]
        \REQUIRE Objective function $f:\M\to \R$, initial guess $x_0\in \M$, retraction $R:\T \M\to \M$, maximal radius $\bar{\mathcal{Q}}>0$, starting radius $\mathcal{Q}_0\in(0,\bar{\mathcal{Q}}]$ and ratio--threshold $\rho'>0$.
        \FOR{$k=1,2,\dots$}
            \STATE Compute  $\xi_k=\argmin_{\xi\in \T_{x_k}\M, \|\xi\|_{x_k}\leq \mathcal{Q}_k}m_k(\xi)$, with $m_k$ as in \eqref{eq:quadratic_model}\\
            \STATE Compute $\hat x = R_{x_k}(\xi_k)$ \\
            \STATE Compute 
            \begin{equation*}
                \rho_k=\frac{f(x_k)-f(\hat x)}{m_k(0)-m_k(p_k)}
            \end{equation*}
            \\
            \STATE If $\rho<\rho'$ take $x_{k+1}=\hat x$, otherwise $x_{k+1}=x_k$\\
            \STATE Update the trust--region radius 
            \begin{equation*}
                \mathcal{Q}_{k+1}=\begin{cases}
                    \frac{1}{4}\mathcal{Q}_k, &\textnormal{if }\rho_k<\frac{1}{4}\\
                    \min\qty{2\mathcal{Q}_k,\bar{\mathcal{Q}}},&\textnormal{if }\rho_k>\frac{3}{4}\textnormal{ and }\|p_k\|_{x_k}=\mathcal{Q}_k\\
                    \mathcal{Q}_k,&\textnormal{otherwise}
                \end{cases}
            \end{equation*}
        \ENDFOR
        \end{algorithmic}
\end{algorithm}
\subsection{Implementation}

The R--SD and R--CG method utilize the (BB method) in Algorithm \ref{alg:BB_method} for line searching. For R--TR we use the implementation provided in the Manopt package. 

We have implemented the symplectic Stiefel manifold in the Manopt framework in order to use the built--in features of the package for optimization. The symplectic Stiefel manifold is not available in the Manopt package for Matlab, but it is available in a package for Julia, Manifolds.jl \cite{Manifolds_jl}. 
The retraction, vector transport, metric, norm etc. have all been implemented according to the theory outlined in Section \ref{sec:background}. We use this structure to perform numerical experiments, since we use the Manopt implementation of R--TR. We also leverage this structure when testing our manual implementation of R--SD and R--CG. 

The algorithms terminate when $\|\grad f(x_k)\|_{x_k}<10^{-6}$, or when the step size $\tau_k$ in R--SD or R--CG becomes smaller than $10^{-11}$.

We have made minor changes to the script \texttt{trustregions.m} from Manopt package to make our manifold implementation compatible with the Manopt framework. All files are available at \url{https://github.com/JensenRasmus/secondorder_spst}.

\section{Numerical experiments}\label{sec:num_exp}

For the computational competition, we consider the following three problems from numerical linear algebra:
\begin{enumerate}
    \item The nearest symplectic matrix problem \cite{bz_symplectic}.
    \item The symplectic eigenvalue problem \cite{Son_21}.
    \item The proper symplectic decomposition \cite{bz_symplectic,bendokat_geom_22}.
\end{enumerate}
All experiments are conducted with Matlab 2023b on a MacBook Air 2023 M2 with 16 GB ram and 512 GB SSD. The experiments are conducted under the fixed random seed \texttt{mt19937ar}. For the last iterate at each run $x^*$, we compute the feasibility measure, $\|(x^*)^+x^*-I_{2k}\|_F$, which quantifies numerically if the solution resides on $\spst(2n,2k)$. 
The value should be close to $0$. Otherwise, the algorithm is compromised by numerical instability. 

We consider R--TR in two variants: R-TR1 uses the true Riemannian Hessian, and R--TR2 uses the approximation given in \eqref{eq:ehess_to_rhess_euc}. We will write R--TR when addressing both variants.

\subsection{The nearest symplectic matrix problem}

We consider the problem of finding a symplectic matrix $U\in\spst(2n,2k)$, such that for a given fixed matrix $A\in \R^{2n\times 2k}$, the objective function $f(U)=\frac{1}{2}\|A-U\|_F^2$ is minimized. We aim to solve the minimization problem
\begin{equation}\label{eq:nearest_spst_prob}
    \min_{U\in \spst(2n,2k)}f(U)=\min_{U\in \spst(2n,2k)}\frac{1}{2}\|A-U\|^2_F.
\end{equation}
The Euclidean gradient and Hessian are
\begin{equation}
    \nabla f(U)=U-A,  \ \ \nabla^2 f(U)[\Delta]=\Delta. 
\end{equation}
The matrix $A$ is generated randomly, and is subsequently normalized, $\frac{A}{\|A\|_F}$. We take $n=1000
$ and $k=10, 50, 100$. The results of the various optimization approaches are presented in Table \ref{tab:nearest_symp_prob}. 
\begin{table*}\centering
\small
    \ra{1,3} 
    \begin{tabular}{@{}rrrrrr@{}}\toprule
        & \multicolumn{5}{c}{R--SD} \\
\cmidrule{2-6}
& Num.iter. & Time (s) & $\|\grad f(x^*)\|_{x^*}$ & Feasibility of $x^*$&$f(x^*)$ \\ \midrule 
$k = 10$ & 68 & 4.84 & $9.13\cdot 10^{-7}$ & $3.87\cdot 10^{-14}$&8.04 \\
$k = 50$ & 114 & 19.93 & $8.96\cdot 10^{-7}$ & $9.66\cdot 10^{-14}$& 46.04\\
$k = 100$& 129 & 45.11 & $7.22 \cdot 10^{-7}$ & $2.19\cdot 10^{-13}$&94.60\\
& \multicolumn{5}{c}{R--CG} \\
\cmidrule{2-6}
& Num.iter. & Time (s) & $\|\grad f(x^*)\|_{x^*}$ & Feasibility of $x^*$&$f(x^*)$\\ \midrule 
$k = 10$ & 192 & 27.07 & $6.53\cdot 10^{-7}$ & $2.29\cdot 10^{-14}$ &8.04\\
$k = 50$ & 227 & 75.30 & $9.88\cdot 10^{-7}$ & $2.83\cdot 10^{-14}$&46.04 \\
$k = 100$ & 288 & 177.19 & $8.45\cdot 10^{-7}$ & $6.36\cdot 10^{-14}$&94.60\\
& \multicolumn{5}{c}{R--TR1} \\
\cmidrule{2-6}
& Num.iter. & Time (s) & $\|\grad f(x^*)\|_{x^*}$ & Feasibility of $x^*$&$f(x^*)$\\ \midrule 
$k = 10$ & 9 & 17.49& $2.28\cdot 10^{-12}$ & $3.36\cdot 10^{-15}$ &8.04\\
$k = 50$ & 10 & 39.49 & $8.51\cdot 10^{-12}$ & $4.39\cdot 10^{-15}$&46.04 \\
$k=100$ & 12 & 50.58 & $2.48\cdot 10^{-7}$ & $6.60\cdot 10^{-15}$&94.60 \\ 
& \multicolumn{5}{c}{R--TR2 } \\
\cmidrule{2-6}
& Num.iter. & Time (s) & $\|\grad f(x^*)\|_{x^*}$ & Feasibility of $x^*$&$f(x^*)$\\ \midrule 
$k = 10$ & 11 & 7.97& $4.07\cdot 10^{-10}$ & $4.23\cdot 10^{-15}$ &8.04\\
$k = 50$ & 10 & 19.72 & $4.53\cdot 10^{-8}$ & $1.17\cdot 10^{-14}$&46.04 \\
$k=100$ & 12 & 36.27 & $1.67\cdot 10^{-8}$ & $1.80\cdot 10^{-14}$&94.60 \\ 
\bottomrule
    \end{tabular}
    \vspace{0.2cm}
    \caption{Results of applying R--SD, R--CG and R-TR on the problem \eqref{eq:nearest_spst_prob} for $k=5,10,20,30$. The numerical optimum $x^*$ is from the last iterate at which numerical convergence was detected. R--SD  outperforms R--CG and R--TR1, when considering the runtime, but the numerical accuracy and feasibility is best for R--TR1. R--TR1 is superior to R--CG with respect to runtime and feasibility. R--TR2 outperforms all the methods with respect to runtime with the exception of R--SD for $k=10$, while maintaining reasonable feasibility.}    \label{tab:nearest_symp_prob}
\end{table*}
From the table it is seen that R--TR2  outperforms R--SD except for $k=10$, and outperforms R--CG and R--TR1 in terms of the wallclock runtime in all the experiments.
The iteration count of R--CG is largest,
and each iteration is more costly than for to R--SD, since the vector transport needs to be computed. As the line search is neither guaranteed nor expected to be exact or to fulfill both of the Wolfe conditions \eqref{eq:armijo_cond}--\eqref{eq:wolfe_2}, there is also no theoretical  reason to expect superior convergence rates over R--SD. The lowest iteration count by far is observed for the R--TR methods.
It also exhibits a smaller total runtime than R--CG, and achieves the highest numerical accuracy, both with respect to the zero-gradient condition and the preservation of the symplectice structure as measured by the feasibility.
However, R--TR1 is not superior to R--SD in terms of the runtime, but R--TR2 is.

\subsection{The symplectic eigenvalue problem} 
\label{sec:num_exp_symp}
This example follows \cite{Gao_2021}. The goal is to solve
\begin{equation}\label{eq:sympl_eig_val}
    \min_{X\in\spst(2n,2p)}f(X)=\min_{X\in\spst(2n,2p)}\tr(X^TAX).
\end{equation}
This objective function is a special case of the Brockett cost function \cite{brockett_89}
\begin{equation*}
    b(X)=\tr(X^TAXN-2BX^T), \quad \text{with } A \text{ sym. pos. def.}, B=0,N=I\in \R^{n\times n}.
\end{equation*}
Minimizing \eqref{eq:sympl_eig_val} is of interest when solving the symplectic eigenvalue problem \cite[Theorem 4.1]{Son_21}.
The minimizer $X^*\in \spst(2n,2p)$ spans a symplectic subspace, and can be used to compute the trailing symplectic eigenvalues and associated symplectic eigenvectors of $A$. It holds that the smallest $p$ symplectic eigenvalues $d_1,\dots,d_p$ are related to the true minimizer $X_{\textnormal{opt}}$ via
\begin{equation*}
    2\sum_{i=1}^p d_i=f(X_{\textnormal{opt}}).
\end{equation*}
For more details, including algorithms, see \cite{Son_21}.

The Euclidean gradient and Hessian of \eqref{eq:sympl_eig_val} are
\begin{equation}
    \nabla f(X)=2AX, \ \ \nabla^2f(X)[\Delta]=2A\Delta.
\end{equation}
From Williamson's theorem \cite{williamson}, we have that any symmetric positive definite matrix $A\in \R^{2n\times 2n}$ admits the decomposition
\begin{equation*}
    S^TAS =\begin{pmatrix}
        D&0\\
        0&D
    \end{pmatrix},
\end{equation*}
where $D=\textnormal{diag}\qty{\mu_1,\dots,\mu_n}$ are the symplectic eigenvalues. These coincide with the positive imaginary parts of the eigenvalues of $JA$. 

We perform the following numerical experiment inspired by \cite{gao2022optimization, Son_21}. 
Let $D=\textnormal{diag}\qty{1,2,\dots,n}\in \R^{n\times n}$, and $S=J_{2n}KL(l,c,d)$, where $K=\begin{pmatrix}
    \Re(U)&-\Im(U)\\
    \Im(U)&\Re(U)
\end{pmatrix}$ with $U\in \C^{n\times n}$ being a unitary matrix obtained from the Q--factor in the QR--decomposition of a random complex matrix. The $\R^{2n\times 2n}$--matrix $L(l,c,d)$ is  
\begin{equation*}
    L(l,c,d)=\begin{pmatrix}
        I_{l-2}& & & & & & & & \\
               & c & & & & & d & \\
               & & c & & & d& & \\
               & & & I_{n-l}& & & &\\
               & & & & I_{l-2} & & & \\
               & & & & & c^{-1} & &   \\
               & & & & & & c^{-1} & \\
               & & & & & & & I_{n-l}
    \end{pmatrix},
\end{equation*}
and is called symplectic Gauss transformation. Finally, define $A=S\begin{pmatrix}
    D&0\\
    0&D
\end{pmatrix}S^T$. We take $p = 5$ in \eqref{eq:sympl_eig_val}, and thus a true minimizer features an objective function value of $f(X_{\textnormal{opt}})=30$. We employ the algorithms \cite[Algorithms 5.2 and 3.1]{Son_21} to numerically recover the (known) smallest 5 symplectic eigenvalues $A$.
For this purpose, the optimizer of \eqref{eq:sympl_eig_val} is used. 

The optimization results are displayed in Table \ref{tab:sympl_eig}.
All three methods yield a numerical solution of acceptable accuracy in terms of the objective function value, but R--SD and R--CG reached a point where the step size became smaller than $10^{-11}$.
This is an alternative stopping criterion in the implementation, which results in the algorithm terminating before the gradient stopping criterion was reached. The runtime is smallest for R--TR2, and largest for R--CG. Considering the norm of the gradient for the last iterate, we see that R--CG has not properly converged. We see that the feasibility violation is smallest for R--TR, and for R--CG it is by far the largest, which is most likely related to the high iteration count.

The symplectic eigenvalues that are computed with the numerical optimizer $x^{*}$ are presented in Table \ref{tab:sympl_eig_vals}.
The worst case absolute error occurs for the eigenvalue $\lambda = 5$ for all methods and is of the order $10^{-12}$ for R--SD, $10^{-8}$ for R--CG, and $10^{-14}$ for R--TR. This is consistent with the numerical accuracy displayed in Table \ref{tab:sympl_eig}.
\begin{table*}\centering
\small
    \ra{1,3} 
    \begin{tabular}{@{}rrrrrr@{}}\toprule
                & Num.iter. & Time (s) & $\|\grad f(x^*)\|_{x^*}$ & Feasibility  & $|f(x^*)-f(X_{\textnormal{opt}})|$\\ \midrule 
R--SD $^\dagger$ & 1148 & 81.59  & $9.70\cdot 10^{-6}$ & $8.98\cdot 10^{-14}$ & $3.95\cdot 10^{-12}$ \\
R--CG $^\dagger$& 4325  & 616.75 & $9.56\cdot 10^{-3}$ & $4.05\cdot 10^{-13}$ & $5.38\cdot 10^{-8}$\\
R--TR1& 41    & 221.73 & $1.33 \cdot 10^{-11}$ & $2.00\cdot 10^{-15}$ & $1.78\cdot 10^{-14}$\\
R--TR2& 20    & 76.74 & $2.80 \cdot 10^{-10}$ & $2.76\cdot 10^{-15}$ & $7.11\cdot 10^{-14}$\\
  \bottomrule
    \end{tabular}
    \vspace{0.2cm}
    \caption{Results of applying R--SD, R--CG and R--TR on the problem \eqref{eq:sympl_eig_val}. The last iterate obtained at each experiment is $x^*$, and $f(x_{\textnormal{opt}})=30$. R--TR2 outperforms R--SD, R--CG and R--TR1 with respect to computation time.  R--CG terminates at a comparably large gradient norm. The feasibility of the result obtained from R--TR1 and R--TR2 is better than the one from R--SD and R--CG. $\dagger$: The algorithm terminated due to the step size being smaller than $10^{-11}$.}   \label{tab:sympl_eig}
\end{table*}
\begin{table*}\centering
\small
    \ra{1,4} 
    \begin{tabular}{@{}rrrrr@{}}\toprule
                & R--SD. & R--CG  & R--TR1 & R--TR2 \\ 
                1.&$1.000000000000044$ & $1.000000000006140$ & $1.000000000000037$ & $1.000000000000015$\\
                2. & $2.000000000000119$ & $2.000000000076801$ & $2.000000000000014$ & $2.000000000000045$\\
                3. & $3.000000000000211$ & $3.000000000186504$ & $3.000000000000015$ & $3.000000000000023$\\ 
                4. & $4.000000000000287$ & $4.000000000330795$ & $4.000000000000025$ & $3.999999999999970$\\
                5. & $5.000000000001303$ & $5.000000026306411$ & $4.999999999999932$ & $4.999999999999988$ \\
  \bottomrule
    \end{tabular}
    \vspace{0.2cm}
    \caption{The computed symplectic eigenvalues resulting from using the last iterate $x^*$ obtained from R--SD, R--CG and R--TR. R--TR is overall seen to produce the most exact symplectic eigenvalues. }    \label{tab:sympl_eig_vals}
\end{table*}

\subsection{The proper symplectic decomposition} 
Computing the proper symplectic decomposition (PSD) of a matrix $S\in \R^{2n\times 2m}$ is a central task in symplectic model order reduction of Hamiltonian systems \cite{Hesthaven17, bendokat_geom_22, Buchfink_22}. The problem can be seen as the symplectic analog to the proper orthogonal decomposition from classical snapshot--based model order reduction.
The latter is obtained from a (truncated) singular value decomposition of the data matrix \cite{Pinnau2008}. The PSD problem is 
\begin{equation}\label{eq:psd}
    \min_{U\in \spst(2n,2k)}f(U)=\min_{U\in \spst(2n,2k)}\|S-UU^+S\|^2_F.
\end{equation}
In practical applications, usually $k\ll n$. The gradient and Hessian are 
\begin{align}
    \nabla f(U)=&-2\qty(\qty(I_{2n}-UU^+)SS^TJ_{2n}^T U J_{2k}-J_{2n}SS^T(I_{2n} - UU^+)^TUJ_{2k} ),\notag \\
    \nabla^2 f(U)[\Delta]=&-2\qty(A(U,\Delta)SS^TJ_{2n}UJ_{2k}+(I_{2n}-UU^+)SS^TJ_{2n}^T\Delta J_{2k})\\
    &+2\qty(J_{2n}SS^TA(U,\Delta)^TUJ_{2k}+J_{2n}SS^T(I_{2n}-UU^+)\Delta J_{2k})\notag,
\end{align}
where $A(U,\Delta)=\D (I-UU^+)[\Delta]=-\Delta U^+-(\Delta U^+)^+$. In the implementation, we have ensured to store and recycle computations of the numerous matrix--matrix products appearing in the expressions for the gradient and the Hessian. We generate a matrix of rank $2r, r=40$, $T \in \spst(2n,2r)$, generate another random matrix $C\in \R^{2r\times 2m}$, normalize it $C=\frac{C}{\|C\|_F}$, and set $S=TC$. We consider unknowns $U\in \spst(2n,2k)$ for $k=10,20,40$. When $k=r=40$ we expect the  minimizer $x^*$ to fulfill $f(x^*)\approx 0$, where clearly $x^*=T$ is a solution.

In Table \ref{tab:psd}, the results show that R--SD outperforms R--CG and R--TR in terms of the runtime, and R--CG is faster than R--TR1 for $k=10$ and $k=40$, but R--TR2 is faster than R--CG in all runs.   For R--SD and R--CG the iteration count and runtime go up when stepping from $k=10$ to $k=10$, but go down for $k=40$. This might be explained in part by the fact that $k=40$ corresponds to the rank of the true solution.

 The numerical feasibility (i.e. solutions staying symplectic) is best for R--TR with R--SD ranking second. For R--CG, it is the lowest by one order of magnitude, but under R--CG, the most iterations and thus the most numerical operations are performed.
 
 The objective function values follow the expected pattern: When $k$ increases, the projection error decreases. For $k=40$, perfect recovery is possible. The output of R--TR comes closest to the ideal projection error of $0$; for R--SD and R--CG the projection error is acceptable but it is four orders of magnitude larger than that of R--TR.
\begin{table*}\centering
\small
    \ra{1,3} 
    \begin{tabular}{@{}rrrrrr@{}}\toprule
        & \multicolumn{5}{c}{R--SD} \\
\cmidrule{2-6}
& Num.iter. & Time (s) & $\|\grad f(x^*)\|_{x^*}$ & Feasibility of $x^*$ & $f(x^*) $ \\ \midrule 
$k = 10$ & 550 & 199.08 & $9.15\cdot 10^{-7}$ & $4.49\cdot 10^{-14}$ &0.127\\
$k = 20$ & 770 & 349.07 & $9.91\cdot 10^{-7}$ & $7.90\cdot 10^{-14}$ & $0.050$\\
$k = 40$& 112 & 58.85 & $9.06 \cdot 10^{-7}$ & $5.08\cdot 10^{-14}$&$3.15\cdot 10^{-9}$\\
& \multicolumn{5}{c}{R--CG} \\
\cmidrule{2-6}
& Num.iter. & Time (s) & $\|\grad f(x^*)\|_{x^*}$ & Feasibility of $x^*$ & $f(x^*)$\\ \midrule 
$k = 10$ & 1186 & 618,25 & $8.87\cdot 10^{-7}$ & $1.12\cdot 10^{-12}$ &$0.127$\\
$k = 20$ & 2527 & 1496.90 & $7.71\cdot 10^{-7}$ & $1.18\cdot 10^{-13}$&$0.051$ \\
$k = 40$ & 418 & 291.55 & $9.63\cdot 10^{-7}$ & $1.91\cdot 10^{-13}$ & $3.53\cdot 10^{-9}$\\
& \multicolumn{5}{c}{R--TR1} \\
\cmidrule{2-6 }
& Num.iter. & Time (s) & $\|\grad f(x^*)\|_{x^*}$ & Feasibility of $x^*$ &$f(x^*) $ \\ \midrule 
$k = 10$ &48 & 889.77& $4.81\cdot 10^{-7}$ & $9.49\cdot 10^{-15}$&$0.127$ \\
$k = 20$ & 54 & 1469.26 & $3.53\cdot 10^{-7}$ & $1.28\cdot 10^{-14}$&$0.053$ \\
$k=40$ & 20 & 1064.06 & $6.40\cdot 10^{-7}$ & $1.52\cdot 10^{-14}$&$2.38\cdot 10^{-13}$ \\ 
& \multicolumn{5}{c}{R--TR2} \\
\cmidrule{2-6 }
& Num.iter. & Time (s) & $\|\grad f(x^*)\|_{x^*}$ & Feasibility of $x^*$ &$f(x^*) $ \\ \midrule 
$k = 10$ &31 & 281.15& $6.11\cdot 10^{-7}$ & $1.60\cdot 10^{-14}$&$0.127$ \\
$k = 20$ & 36 & 496.34 & $6.74\cdot 10^{-7}$ & $2.37\cdot 10^{-14}$&$0.053$ \\
$k=40$ & 12 & 116.57 & $5.56\cdot 10^{-7}$ & $1.42\cdot 10^{-14}$&$4.00 \cdot 10^{-13}$ \\ \bottomrule
    \end{tabular}
    \vspace{0.2cm}
    \caption{Results of applying R--SD, R--CG and R--TR on the problem \eqref{eq:psd} for $k=10,20,40$. $x^*$ is the last iterate obtained at each experiment. R--SD outperforms R--CG and R--TR for the experiments at hand, when considering the runtime.}    \label{tab:psd}
\end{table*}

\section{Discussion and conclusion}\label{sec:discussion}
% \RZ{The discussion needs a major rewrite. I denne form er det en genfort{\ae}lling. You need to discuss and summarize the main findings.
% I don't want to see any "we have" in this section ;-).
% A personal view point is OK, though. For example: "According to our experiments, simple R-SD outperforms its competitors in terms of the runtime. However, the highest numerical accuracy is achieved with R-TR.
% As of now, R-CG is not to be recommended for use on the symplectic Stiefel manifold. Since the iteration count of R--TR is low, major improvements could be expected, if one succees in making the single iterations computationally more efficient....}
% \RJ{Affirmative -- I expected this.}\\
% \textbf{Goals:}
% \begin{itemize}
%     \item Present principles, relationships and generalizations shown by the results
% \item Point out exceptions or lack of correlations. Define why you think this is so.
% \item Show how your results agree or disagree with previously published works
% \item Discuss the theoretical implications of your work as well as practical applications
% \item State your conclusions clearly. Summarize your evidence for each conclusion.
% \item Discuss the significance of the results
% \end{itemize}

% \RJ{Under construction:}

In the numerical experiments conducted in this study, R--TR2 outperformed R--SD, R--CG and R--TR1 with respect to runtime in the nearest symplectic matrix problem and the symplectic eigenvalue problem. With respect to stability and numerical accuracy, R--TR1 and R--TR2 produced the best results. On the one hand, this is as expected, because R--TR makes use of second--order or approximate second--order information. On the other hand, one cannot conclude that R--CG and R--SD suffer more from numerical errors 
 per iteration. The number of iterations for these methods is much larger than that for R--TR, so that  numerical errors accumulate over a longer iteration history. Of course, R--TR could be further improved, if one succeeds in making a single iterations in R--TR computationally more efficient.
 In our experiments, R--TR  always converged properly, whereas for R--SD and R--CG, there was a case, where the step size became so small
 that the algorithm stopped without reaching a point with near-zero gradient.

Even though it approximates second--order information, R--CG does not converge faster than R--SD in our experiments, and the iteration count is large. We believe that this indicates that the line search is not fulfilling the requirements leading to favourable convergence patterns for R--CG. This is in contrast to what the authors of \cite{Yamada_Sato23} observed. We emphasize that we have taken care to ensure that the numerical competition between the methods considered in this study is fair. The comparison with results from other publications should be made with caution due to the different user parameters that occur in optimization algorithms, not least among them the numerical stopping criterion and the target accuracy. 

The retraction used in the numerical experiments approximates the geodesics on the manifold to comparably high accuracy. Yet, to compute this retraction requires more effort than for the simple retraction \eqref{eq:simple_ret}, and the runtime of R--SD observed in this work is generally much slower than the runtime that is reported for R--SD in the works \cite{bz_symplectic,Gao_2021,gao2022optimization} for similar experiments. These works, exclusively or partly, make use of the simple retraction \eqref{eq:simple_ret}. This supports the usual observation that the closeness of the retraction to the Riemannian exponential plays a subordinate role in iterative optimization schemes and that numerical efficiency should prioritized.

The discordance between the the experiments conducted in this work and those appearing in the works \cite{Gao_2021,gao2022optimization, Yamada_Sato23} can also be due to the different choice of Riemannian metrics, as different metrics gives rise to different formulas for the Riemannian gradient.

\appendix 
\section{Differentiation of vector fields on Riemannian manifolds}\label{sec:appendix}
In this appendix we will recall the background theory for the results presented in Section \ref{sec:rie_hess}. These results can be found in any standard textbook on Riemannian geometry, for example \cite{docarmo92}. The interested reader may also consult \cite{NBoumal}.

Let $\M$ be a Riemannian manifold with metric $\inp{\cdot}{\cdot}_x$ for $x\in \M$ and let $\nabla$ denote the Riemannian connection and let $\mathfrak{X}(\M)$ denote the space of vector fields on $\M$. Below we summarize the properties of the Levi--Civita connection which allows to differentiate a vector field along directions indicated by another vector field.

\begin{theorem}[Levi--Civita connection]
    The Levi--Civita connection is the u\-nique $\R$-bilinear smooth map on $(\M, \inp{\cdot}{\cdot}_x)$
    \begin{equation*}
        \nabla: \mathfrak{X}(\M)\times \mathfrak{X}(\M)\to \mathfrak{X}(\M), (X,Y)\mapsto \nabla_X Y
    \end{equation*}
    such that 
    %which is  $\nabla_{aU+bV}(\lambda X+\mu Y)=a\lambda \nabla_UX+b\mu\nabla_VY$ where $U,V,X,Y\in \mathfrak{X}(\M)$ and $a,b,\lambda,\mu$, and satisfies for $h\in C^\infty(\M,\R)$
    \begin{enumerate}
        \item $\nabla_{hX}Y=h\nabla_X Y$ 
        \item $\nabla_X(hY)=X(h)+h\nabla_X Y$ (Leibnitz--rule)
        \item $\nabla_X Y-\nabla_Y X=[X,Y]$ (torsion free),
        \item $Z\inp{X}{Y}=\inp{\nabla_ZX}{Y}+\inp{X}{\nabla_Z Y}$ (metric)
    \end{enumerate}
    for $X,Y,Z\in \mathfrak{X}(\M)$ and $h\in C^\infty(\M,\R)$.
\end{theorem}
We will make use of the Christoffel formalism indicated in \cite{Edelman}. Briefly; a smooth vector field on $\M$, $x\mapsto V(x)\in \T_x\M$ can be written in the local coordinates as 
\begin{equation*}
    V(x)=\sum_{i=1}^n \alpha_i(x)\partial_i|_x,
\end{equation*} 
where $\partial_1|_x,\dots,\partial_n|_x$ are the canonical basis of the tangent space $\T_x\M$ in the geometric sense. The \textit{Christoffel symbols} are scalar coefficient functions $\Gamma_{i,j}^k:\M\to \R, x\to \Gamma_{i,j}^k(x)$ so that in local coordinates, the covariant derivative of $X(x)=\sum_{i=1}^n \alpha_i(x)\partial_i|_x \in \mathfrak{X}(\M)$ along $Y(x)=\sum_{i=1}^n \beta_i(x)\partial_i|_x\in \mathfrak{X}(\M)$, which is (suppressing the variable $x$) 
\begin{equation*}
    \nabla_XY=\sum_{i=1}^n\sum_{j=1}^n \alpha_i\partial_i(\beta_j)\partial_j+\alpha_i\beta_j\nabla_{\partial_i}\partial_j = 
    \sum_{i=1}^n\sum_{j=1}^n \alpha_i\partial_i(\beta_j)\partial_j+\alpha_i\beta_j
    \sum_{k=1}^n\Gamma_{i,j}^k\partial_k.
\end{equation*} 
In other words, the Christoffel symbols are in some sense `corrector terms', ensuring that $\nabla_XY\in \T_x\M$, whenever $X(x)\in \T_x\M$. We will later be interested in the special case when $X=c'(t)$, for a smooth curve $c:\R\supset I\to \M, t\mapsto(\gamma_1(t),\dots,\gamma_n(t))$ with the associated directional derivative along the coordinate lines, $c'(t)=\sum_{i=1}^n \gamma_i'(t)\partial_i|_{c(t)}$. This is \textit{the covariant derivative along a curve}, and is written as $\icov X(t)=\nabla_{c'(t)}X(t)$. Using the same representation for $X(t)$ as earlier, $\icov X(t)$ is represented in the basis $\partial_1,\dots,\partial_n$ as 
\begin{equation}\label{eq:icov_formula}
    \icov X(t)=\alpha'(t)+\Gamma(\alpha(t),\gamma'(t)),
\end{equation}
where 
\begin{equation}\label{eq:christ_form}
    \Gamma(u,v)=\begin{pmatrix}
        u^T\Gamma^1 v\\
        \vdots\\
        u^T\Gamma^m v
    \end{pmatrix},
\end{equation} 
is the so-called \textit{Christoffel function} and thus $\alpha_k'(t)+\Gamma^k(\alpha(t),\gamma'(t))$ are the coefficient for the basis element $\partial_k$ at $c(t)$. 

If $\gamma$ is a geodesic, then $\icov \gamma'(t)\equiv 0$ and thus \eqref{eq:icov_formula} in this case reduces to 
\begin{equation}\label{eq:geod_christ_rel}
    0=\gamma''(t)+\Gamma(\gamma'(t),\gamma'(t)),
\end{equation}
which is also stated in \cite{Edelman}. What is important is that the Christoffel symbols depend solely on the local coordinates and the Riemannian metric, and thus when then Christoffel function have been identified, we can use it when the curve of interest is not necessarily a geodesic. By applying the polarization identity
\begin{equation}\label{eq:polar}
    \Gamma(\Delta,\xi)=\frac{1}{4}\qty(\Gamma(\Delta+\xi,\Delta+\xi)-\Gamma(\Delta-\xi,\Delta-\xi)),
\end{equation}
we can recover the Christoffel function for two different inputs, as required in the genral formula  \eqref{eq:icov_formula}. 

Finally, for $X(t)=\grad f(c(t))$, $c(0)=x,c'(0)=\Delta$ for a smooth map $f:\M\to \R$, we have that the Riemannian Hessian is computed as 
\begin{equation}\label{eq:rie_hess_def}
    \eval{\icov \grad f(c(t))}_{t=0}=\hess f(x)[\Delta].
\end{equation}

\section{Derivation of the low--rank formula for the retraction \eqref{eq:main_ret}} \label{sec:appendixB}
In this appendix we will derive a low--rank formula for the Cayley retreaction presented in this paper.

We  employ the Sherman--Morrison--Woodbury formula
\begin{equation*}
    (A+UCV)^{-1}=A^{-1}-A^{-1}U(C^{-1}+VA^{-1}U)^{-1}VA^{-1}
\end{equation*}
and consider the retraction as a curve $t\mapsto R_U(t\Delta)$,
\begin{equation}
    R_U(t\Delta)=\cay\qty(\frac{t}{2}\qty(\bar\Omega(\Delta)-\bar\Omega(\Delta)^T))\cay\qty(\frac{t}{2}\bar\Omega(\Delta)^T)U.
\end{equation}
First we tackle the rightmost part of $R_U(t\Delta)$, 
\begin{align}\label{eq:efficient_cay_1}
    \cay\qty(\frac{t}{2}XY^T)U&=\qty(I_{2n}+\frac{t}{2}XY^T)\qty(I_{2n}-\frac{t}{2}XY^T)^{-1}U\notag\\
    &=\qty(I_{2n}+\frac{t}{2}XY^T)\qty(I_{2n}+\frac{t}{2}X\qty(I_{4k}-\frac{t}{2}Y^TX)^{-1}Y^T)U\notag\\
    &=U+\qty(\frac{t}{2}X\qty(I_{4k}-\frac{t}{2}Y^TX)+\frac{t^2}{4}XY^TX+\frac{t}{2}X)\notag \\ 
    &\ \ \ \ \ \ \ \ \ \  \cdot \qty(I_{4k}-\frac{t}{2}Y^TX)^{-1}Y^TU \notag\\
    &=U+tX\qty(I_{4k}-\frac{t}{2}Y^TX)^{-1}Y^TU.
\end{align}
Because $\bar \Delta=UA+H$ (this can be seen similarly to how the authors do in \cite[pp. 13]{bz_symplectic}), we have $A=U^+\bar \Delta, H=\bar \Delta-UA$ and thus one obtains
\begin{align*}
    Y^TX=\begin{bmatrix}
        \frac{1}{2}A&-I_{2k}\\
        H^+H-\frac{1}{4}A^2&\frac{1}{2}A
    \end{bmatrix},
\end{align*}
and thus
\begin{equation}\label{eq:to_be_inverted_schur}
    I_{4k}-\frac{t}{2}Y^TX=\begin{bmatrix}
        I_{2k}-\frac{t}{4}A&\frac{t}{2}I_{2k}\\
        -\frac{t}{2}(H^+H-\frac{1}{4}A^2)&I_{2k}-\frac{t}{4}A
    \end{bmatrix}.
\end{equation} 
To invert the matrix in \eqref{eq:to_be_inverted_schur}, we use the method of block--inversion by Schur complements \cite{schur_and_app},
\begin{equation}\label{eq:schur_block_formula}
    M=\begin{bmatrix}
        \hat A&\hat B\\
        \hat C&\hat D
    \end{bmatrix}\quad \Rightarrow \quad
    M^{-1}=\begin{bmatrix}
        S^{-1}&-S^{-1}\hat B\hat D^{-1}\\
    -\hat D^{-1}\hat CS^{-1}&\hat D^{-1}+\hat D^{-1}\hat CS^{-1}\hat B\hat D^{-1}
    \end{bmatrix},
\end{equation}
where the Schur complement $S$ is with respect to the lower block $\hat D$,
$
    S=\hat A-\hat B\hat D^{-1}\hat C.
$

In the case at hand,
\begin{align*}
    S&=I_{2k}-\frac{t}{4}A+\frac{t^2}{4}\qty(I_{2k}-\frac{t}{4}A)^{-1}\qty(H^+H-\frac{1}{4}A^2)\\
    &=\qty(I_{2k}-\frac{t}{4}A)^{-1}\qty(\qty(I_{2k}-\frac{t}{4}A)^2+\frac{t^2}{4}\qty(H^+H-\frac{1}{4}A^2))\\
    &=\qty(I_{2k}-\frac{t}{4}A)^{-1}\Theta,\quad 
    \text{ where }
    \Theta=\qty(I_{2k}-\frac{t}{2}A+\frac{t^2}{4}H^+H).
\end{align*}
 %When the blocks are calculated, it is beneficial to use the relationship $\Theta=\qty(\qty(I_{2k}-\frac{t}{4}A)^2+\frac{t^2}{4}\qty(H^+H-\frac{1}{4}A^2))$. 

Combined with \eqref{eq:schur_block_formula}, one obtains
\begin{equation*}
    \qty(I_{4k}-\frac{t}{2}Y^TX)^{-1}=\begin{bmatrix}
        \Theta^{-1}\qty(I_{2k}-\frac{t}{4}A)&-\frac{t}{2}\Theta^{-1} \\
        -\frac{2}{t}\qty(\qty(I_{2k}-\frac{t}{4}A)\Theta^{-1}(I_{2k}-\frac{t}{4}A)-I_{2k})&\qty(I_{2k}-\frac{t}{4}A)\Theta^{-1}
    \end{bmatrix},
\end{equation*}
which leads to the formula
\begin{align}\label{eq:eff_cay_small}
    \cay\qty(\frac{t}{2}\Omega(\Delta)^T)U&=U+t\begin{bmatrix}
        (I_{2n}-\frac{1}{2}UU^+)\bar \Delta&-U
    \end{bmatrix}\begin{bmatrix}
        \Theta^{-1}\\
        -\frac{2}{t}\qty(I_{2k}-\frac{t}{4}A)\Theta^{-1}+\frac{2}{t}I_{2k}
    \end{bmatrix}\notag \\
    &=-U+(tH+2U)\Theta^{-1},
\end{align}
For the left Cayley factor in the retraction formula \eqref{eq:main_ret}, we proceed in a similar manner and obtain  
\begin{equation}\label{eq:cay_large}
    \cay\qty(\frac{t}{2}\hat X\hat Y^T)=I_{2n}+t\hat X\qty(I_{8k}-\frac{t}{2}\hat Y^T\hat X)^{-1}\hat Y^T.
\end{equation}
Due to the complicated matrix appearing under the inverse
\begin{equation*}
    \hat Y^T\hat X=\begin{bmatrix}
        X^TY&-X^TX\\
        Y^TY&-Y^TX
    \end{bmatrix},
\end{equation*}
we have not obtained any fruitful results via inversion by Schur compliments, and we therefore accept computing $\qty(I_{8k}-\frac{t}{2}\hat Y^T\hat X)^{-1}$. We therefore have the resulting low--rank formula for the retraction 
\begin{equation*}
    \begin{aligned}
    R_U(t\Delta)=&\qty(I_{2n}+t\hat X\qty(I_{8k}-\frac{t}{2}\hat Y^T\hat X)^{-1}\hat Y^T)\\
    &\cdot \qty(-U+(tH+2U)\qty(I_{2k}-\frac{t}{2}A+\frac{t^2}{4}H^+H)^{-1}).
\end{aligned}
\end{equation*}

\section{Basic calculation rules} Here we mention three useful calculation rules for matrices on $\spst(2n,2k)$, which all follow directly from the definition of the manifold. 
\begin{lemma}\label{lam:c1}
    Let $U\in \spst(2n,2k)$ and $v\in \R^{2n\times 2k}$. The following properties hold 
    \begin{itemize}
        \item $((U^TU)^{-1})^+=((U^TU)^+)^{-1}$
        \item $(v^+U+U^+v)^+=v^+U+U^+v$
        \item $U^+(U^T)^+=(U^TU)^+$
    \end{itemize} 
\end{lemma}

\bibliographystyle{siamplain}
\bibliography{mybib}

\newpage

\end{document}

%% file: preamble.tex
\usepackage[utf8]{inputenc}	

\usepackage[T1]{fontenc}

\usepackage{graphicx}

\usepackage{physics}
\usepackage{amssymb}

\usepackage[makeroom]{cancel}
\usepackage[normalem]{ulem}

\usepackage{color}

\newsiamremark{remark}{Remark}
\newsiamremark{observation}{Observation}

\usepackage{url}
\newcommand{\N}{\mathbb{N}}
\newcommand{\R}{\mathbb{R}} % the reals
\newcommand{\C}{\mathbb{C}} % the complex numbers

\usepackage{mathtools}

\DeclarePairedDelimiterX{\inp}[2]{\langle}{\rangle}{#1, #2}

\usepackage{mathrsfs}

\DeclareMathOperator{\D}{\textnormal{D}}
\DeclareMathOperator{\M}{\mathcal{M}}
\DeclareMathOperator{\V}{\mathcal{V}}

\renewcommand{\H}{\mathcal{H}}
\DeclareMathOperator{\hess}{\textnormal{Hess}}

\renewcommand{\grad}{\textnormal{grad}}

\DeclareMathOperator{\icov}{\frac{\D}{\dd t}}

\renewcommand{\skew}{\textnormal{skew}}

\renewcommand{\sp}{\textnormal{Sp}}
\DeclareMathOperator{\spst}{\textnormal{SpSt}}
\DeclareMathOperator{\hamil}{\mathfrak{s}\mathfrak{p}}

\DeclareMathOperator{\T}{\textnormal{T}}
\DeclareMathOperator{\cay}{\textnormal{Cay}}

\DeclareMathOperator{\tran}{\mathcal{T}}

\makeatletter
\renewcommand*\env@matrix[1][*\c@MaxMatrixCols c]{%
	\hskip -\arraycolsep
	\let\@ifnextchar\new@ifnextchar
	\array{#1}}
\makeatother

\DeclareMathOperator*{\argmin}{arg\,min}